\documentclass[openacc]{rsproca_new}

\newtheorem{theorem}{\bf Theorem}[section]

\newtheorem{corollary}{\bf Corollary}[section]

\newtheorem{proposition}{\bf Proposition}[section]
\newtheorem{lemma}{\bf Lemma}[section]
\newtheorem{remark}{\bf Remark}[section]

\usepackage{enumitem}
\usepackage{amsthm}
\usepackage{amsmath}
\usepackage{amssymb}
\usepackage{mathrsfs}
\usepackage{float}
\allowdisplaybreaks

\newcommand\scaleddot{\scalebox{.89}{.}}
\makeatletter
\renewcommand{\dddot}[1]{%
  {\mathop{\kern\z@#1}\limits^{\makebox[0pt][c]{\vbox to-2\ex@{\kern-\tw@\ex@\hbox{\normalfont\scaleddot\kern-0.5pt\scaleddot\kern-0.5pt\scaleddot}\vss}}}}}

\renewcommand{\AA}{\mathbf{A}}
\newcommand{\BB}{\mathbf{B}}
\newcommand{\CC}{\mathbf{C}}
\newcommand{\FF}{\mathbf{F}}
\newcommand{\GG}{\mathbf{G}}
\newcommand{\HH}{\mathbf{H}}

\newcommand{\NN}{\mathbf{N}}

\newcommand{\cc}{\mathbf{c}}
\newcommand{\ff}{\mathbf{f}}
\renewcommand{\gg}{\mathbf{g}}
\newcommand{\hh}{\mathbf{h}}
\newcommand{\jj}{\mathbf{j}}
\newcommand{\kk}{\mathbf{k}}
\newcommand{\uu}{\mathbf{u}}
\newcommand{\vv}{\mathbf{v}}
\newcommand{\nn}{\mathbf{n}}

\newcommand{\dzeta}{\, \mathrm{d}\zeta}

\newcommand{\udl}[1]{\underline{\smash{#1}}}

\DeclareMathOperator{\sech}{sech}

\let\div\relax
\DeclareMathOperator{\div}{div}
\DeclareMathOperator{\grad}{grad}
\DeclareMathOperator{\curl}{curl}

\renewcommand{\i}{\mathrm{i}}
\newcommand{\e}{\mathrm{e}}

\begin{document}

\title{A variational formulation for steady surface water waves on a Beltrami flow}

\author{
M. D. Groves$^{1}$ and J. Horn$^{1}$}

\address{$^{1}$ Fachrichtung Mathematik, Universit\"at des Saarlandes, Postfach 151150, 66041 Saarbr\"ucken, Germany\\}

\subject{differential equations, fluid mechanics}

\keywords{Beltrami flows, water waves, calculus of variations}

\corres{M. D. Groves\\
\email{groves@math.uni-sb.de}}

\begin{abstract}
This paper considers steady surface waves `riding' a Beltrami flow
(a three-dimensional flow with parallel velocity and vorticity fields). It is demonstrated that
the hydrodynamic problem can be formulated as two equations for two scalar functions of
the horizontal spatial coordinates, namely the elevation $\eta$ of the free surface and the potential $\Phi$
defining the gradient part (in the sense of the Hodge-Weyl decomposition) of the horizontal component
of the tangential fluid velocity there. These equations are written in terms of a nonlocal operator $H(\eta)$
mapping $\Phi$ to
the normal fluid velocity at the free surface, and are shown to arise from a variational principle.
In the irrotational limit the equations reduce to the Zakharov-Craig-Sulem formulation of the classical
three-dimensional steady water-wave problem, while $H(\eta)$ reduces to the familiar Dirichlet-Neumann operator.
\end{abstract}

\begin{fmtext}

\section{Introduction}\label{sec:introduction}

\subsection{The main results}

Consider an incompressible perfect fluid of unit density
occupying a three-dimensional domain bounded below by a rigid horizontal plane and above
by a free surface. A \emph{steady water wave} is a fluid flow of this kind in which both the velocity field and free-surface profile
are stationary with respect to a uniformly (horizontally) translating frame of reference. Working in this frame of
reference, suppose that the fluid domain is
$D_\eta=\{(x,y,z): -h < y < \eta(x,z)\}$ (so that the free surface is the graph $S_\eta$
of an unknown function $\eta$), and
the flow is a \emph{(strong) Beltrami flow} whose velocity and vorticity fields $\uu$ and $\curl \uu$ are parallel, so
that $\curl \uu = \alpha \uu$ for some fixed constant $\alpha$.
The hydrodynamic problem is to solve the equations
\end{fmtext}
\maketitle
\begin{align}
& \parbox{10cm}{$\div \uu = 0$}\mbox{in $D_\eta$,} \label{pre-Beltrami1}\\
& \parbox{10.cm}{$\curl \uu = \alpha \uu$}\mbox{in $D_\eta$,} \label{pre-Beltrami2}\\
& \parbox{10cm}{$\uu\cdot\jj = 0$}\mbox{at $y=-h$,} \label{pre-Beltrami3}\\
& \parbox{10cm}{$\uu\cdot\nn = 0$}\mbox{at $y=\eta$,} \label{pre-Beltrami4}\\
& \parbox{10cm}{$\displaystyle\tfrac{1}{2}|\uu|^2 + g \eta - \sigma
\!\left(\frac{\eta_x}{(1+|\nabla\eta|^2)^\frac{1}{2}}\right)_{\!\!\!x}-\sigma\!\left(\frac{\eta_z}{(1+|\nabla\eta|^2)^\frac{1}{2}}\right)_{\!\!\!z}=\tfrac{1}{2}|\cc|^2$}\mbox{at $y=\eta$} \label{pre-Beltrami5},
\end{align}
where $\nn$ denotes the (upward-pointing) unit normal vector at $S_\eta$, $\jj=(0,1,0)$, $\nabla=(\partial_x,\partial_z)^\mathrm{T}$ and the physical constants $g$, $\sigma$, $\cc=(c_1,c_3)^\mathrm{T}$
are respectively the acceleration due to gravity, the coefficient of surface tension and the wave velocity;
the pressure $p$ in the fluid
is recovered using the formula $p(x,y,z)=-\tfrac{1}{2}|\uu(x,y,z)|^2-gy$ (the variables $\uu$ and $p$ automatically solve
the stationary Euler equation).
 Equations \eqref{pre-Beltrami4} and \eqref{pre-Beltrami5} are referred to as respectively the \emph{kinematic} and
 \emph{dynamic} boundary conditions at the free surface.
It is natural to write $\uu$ as a perturbation of the trivial solution
\begin{equation}
\eta^\star = 0, \qquad \uu^\star=c_1\begin{pmatrix} \cos \alpha y \\0 \\ \sin \alpha y \end{pmatrix}
+c_3 \begin{pmatrix}- \sin \alpha y \\ 0 \\ \cos \alpha y \end{pmatrix}
\label{ABC flow}
\end{equation}
 of \eqref{pre-Beltrami1}--\eqref{pre-Beltrami5}, so that $\vv=\uu-\uu^\star$ satisfies the equations
 \begin{align}
& \parbox{10.5cm}{$\div \vv = 0$}\mbox{in $D_\eta$,} \label{Beltrami 1}\\
& \parbox{10.5cm}{$\curl \vv = \alpha \vv$}\mbox{in $D_\eta$,} \label{Beltrami 2}\\
& \parbox{10.5cm}{$\vv\cdot\jj = 0$}\mbox{at $y=-h$,} \label{Beltrami 3}\\
& \parbox{10.5cm}{$\vv\cdot\nn + \uu^\star \cdot \nn= 0$}\mbox{at $y=\eta$,} \label{Beltrami 4}\\
& \parbox{10.5cm}{$\displaystyle\tfrac{1}{2}|\vv|^2 + \vv \cdot \uu^\star + g \eta - \sigma
\!\left(\frac{\eta_x}{(1+|\nabla\eta|^2)^\frac{1}{2}}\right)_{\!\!\!x}-\sigma\!\left(\frac{\eta_z}{(1+|\nabla\eta|^2)^\frac{1}{2}}\right)_{\!\!\!z}=0$}\mbox{at $y=\eta$.}\label{Beltrami 5}
\end{align}
This paper considers solutions $(\eta,\vv)$ of \eqref{Beltrami 1}--\eqref{Beltrami 5} which are evanescent as
$|(x,z)| \rightarrow \infty$ and therefore represent localised waves `riding' the trivial flow \eqref{ABC flow}.
 
For $\alpha=0$ (and $\uu^\star=(c_1,0,c_3)^\mathrm{T}$)
equations \eqref{Beltrami 1}--\eqref{Beltrami 5}
reduce to the classical three-dimensional irrotational steady water-wave problem,
which is usually handled by writing $\vv=\grad \phi$, where $\phi$ is a harmonic scalar potential, so that
\eqref{Beltrami 1}, \eqref{Beltrami 2} are automatically satisfied. In fact it is possible to formulate this problem
in terms of the variables $\eta$ and $\xi=\phi|_{y=\eta}$ (Zakharov \cite{Zakharov68}, Craig \& Sulem
\cite{CraigSulem93}). Consider the
variational principle
$$\delta \mathcal{L}_0(\eta,\xi)=0,$$
where
$$\mathcal{L}_0(\eta,\xi)=\int_{D_\eta} \tfrac{1}{2}|\grad \phi|^2
+ \int_{{\mathbb R}^2}\bigg( -\eta\,\cc.\nabla (\phi|_{y=\eta})  +
\tfrac{1}{2}g\eta^2 + \sigma\big((1+|\nabla\eta|^2)^\frac{1}{2}-1\big)\bigg)$$
and $\phi$ is the unique harmonic function with $\phi_n|_{y=-h}=0$ and $\phi|_{y=\eta} = \xi$ (so that $\vv=\grad \phi$
satisfies \eqref{Beltrami 1}--\eqref{Beltrami 3});
the Euler-Lagrange equations for ${\mathcal L}_0(\eta,\xi)$ recover the boundary conditions
at the free surface (see Luke \cite{Luke67}). In the Zakharov-Craig-Sulem formulation a
\emph{Dirichlet-Neumann operator} $G(\eta)$ defined by $G(\eta)\xi=\grad \phi|_{y=\eta}\cdot \NN$ is introduced, where
$\NN=(-\eta_x,1,-\eta_z)^\mathrm{T}$ (so that $\nn=\NN/|\NN|$).  One finds that\pagebreak
$$\mathcal{L}_0(\eta,\xi) = \int_{{\mathbb R}^2} \bigg(\tfrac{1}{2}\xi G(\eta)\xi -\eta\, \cc.\nabla \xi
+ \tfrac{1}{2}g\eta^2 + \sigma\big((1+|\nabla\eta|^2)^\frac{1}{2}-1\big)\bigg)$$
and that its Euler-Lagrange equations can be written as
\begin{align*}
& G(\eta)\xi +\cc.\nabla \eta=0, \\
& \tfrac{1}{2}|\nabla\xi|^2 -\frac{(G(\eta)\xi+\cc.\nabla \eta+\nabla\eta\cdot\nabla\xi)^2}{2(1+|\nabla\eta|^2)} \\
& \qquad\mbox{}
-\cc.\nabla \xi
+g\eta- \sigma\!\left(\frac{\eta_x}{(1+|\nabla\eta|^2)^\frac{1}{2}}\right)_{\!\!\!x}-\sigma\!\left(\frac{\eta_z}{(1+|\nabla\eta|^2)^\frac{1}{2}}\right)_{\!\!\!z}= 0,
\end{align*}
which are readily confirmed to be equivalent to the boundary conditions at the free surface (with $\vv=\grad\phi$).
 
This paper presents a generalisation of the Zakharov-Craig-Sulem formulation to the case $\alpha \neq 0$.
The velocity field $\vv$ is represented by a solenoidal vector potential $\AA$ with $\curl \curl \AA = \alpha\,\curl \AA$
and $\AA \wedge \jj|_{y=-h}=0$, so that $\vv=\curl \AA$
automatically satisfies \eqref{Beltrami 1}--\eqref{Beltrami 3}; note that $\uu^\star=\curl \AA^{\!\star}$, where
$$\AA^{\!\star} = \frac{c_1}{\alpha} \begin{pmatrix} \cos \alpha y-1 \\ 0 \\ \sin \alpha y \end{pmatrix}
+\frac{c_3}{\alpha}\begin{pmatrix} -\sin \alpha y \\0 \\ \cos \alpha y-1 \end{pmatrix}.$$
Let $\FF_\parallel$ denote the horizontal component of the tangential part of
a vector field $\FF$ at the free surface, so that $\FF_\parallel = \FF_\mathrm{h}+F_2\nabla \eta\big|_{y=\eta}$,
where $\FF_\mathrm{h}=(F_1,F_3)^\mathrm{T}$,
and write, according to the Hodge-Weyl decomposition for vector fields in two-dimensional free space (see below),
\begin{equation}
\vv_\parallel = \nabla \Phi + \nabla^\perp \Psi, \label{HW decomposition of vv}
\end{equation}
where $\Phi = \Delta^{-1} (\nabla \cdot \vv_\parallel)$, $\Psi=\Delta^{-1}(\nabla^\perp\cdot\vv_\parallel)
=-\Delta^{-1}(\nabla \cdot \vv_\parallel^\perp)$ and $\Delta^{-1}$ is the two-dimensional Newtonian potential.
In Section \ref{VP} it is shown that the hydrodynamic problem
can be formulated in terms of the variables $\eta$ and $\Phi$. The Euler-Lagrange equations for the
variational principle
\begin{equation}
\delta \mathcal{L}(\eta,\Phi)=0, \label{Main VP}
\end{equation}
where
\begin{align}
\mathcal{L}(\eta,\Phi)&=\int_{D_\eta}\!\! \bigg(\tfrac{1}{2} |\curl \AA|^2 - \tfrac{1}{2}\alpha\AA\cdot\curl \AA\bigg)
+\int_{{\mathbb R}^2} \!\!\bigg(\!\!- \tfrac{1}{2}\alpha \nabla \Delta^{-1}(\nabla\cdot\AA^{\!\perp}_\parallel)\cdot\AA_\parallel -\nabla \Phi\cdot\AA^{\!\star\!\perp}_\parallel\bigg)\nonumber \\
& \qquad\mbox{}  +\int_{{\mathbb R}^2}\!\!\bigg(
\Gamma(\eta)+ \tfrac{1}{2}g\eta^2+ \sigma\big((1+|\nabla\eta|^2)^\frac{1}{2}-1\big)\bigg),
\label{Defn of L} \\
\Gamma(\eta) & = -\tfrac{1}{2}\alpha\nabla\Delta^{-1}(\nabla\cdot\AA^{\!\star\!\perp}_\parallel)\cdot\AA^{\!\star}_\parallel
+ \frac{|\cc|^2}{2\alpha}(\sin \alpha \eta - \alpha \eta)\nonumber
\end{align}
and $\AA$ is the unique solution of the boundary-value problem
\begin{align}
& \parbox{7.25cm}{$\curl \curl \AA = \alpha\, \curl\AA$}\mbox{in $D_\eta$,} \label{A BVP 1}\\
& \parbox{7.25cm}{$\div \AA = 0$}\mbox{in $D_\eta$,} \label{A BVP 2}\\
& \parbox{7.25cm}{$\AA \wedge \jj = \mathbf{0}$}\mbox{at $y=-h$,} \label{A BVP 3}\\
& \parbox{7.25cm}{$\AA\cdot\nn = 0$}\mbox{at $y=\eta$,} \label{A BVP 4}\\
& \parbox{7.25cm}{$(\curl \AA)_\parallel = \nabla \Phi - \alpha \nabla^\perp \Delta^{-1} (\nabla\cdot\AA^{\!\perp}_\parallel)
$}\mbox{at $y=\eta$}, \label{A BVP 5}
\end{align}
recover the boundary conditions at the free surface (with $\vv=\curl \AA$);
the existence and uniqueness of the solution to the above boundary-value problem for small values of $|\alpha|$
is demonstrated by functional-analytic methods in Sections \ref{FA}\ref{FA for BVP} and \ref{FA}\ref{FA Analyticity of GDNO}. (Observe that\pagebreak
$$\Psi = -\Delta^{-1}(\nabla\cdot\vv^{\perp}_\parallel)=-\Delta^{-1}(\curl\vv\cdot\NN\big|_{y=\eta})
=-\alpha \Delta^{-1}( \vv\cdot\NN\big|_{y=\eta})=- \alpha \Delta^{-1} (\nabla\cdot\AA^{\!\perp}_\parallel),$$
in which the vector identity $\curl \FF\cdot\NN\big|_{y=\eta}=\nabla \cdot \FF^\perp_\parallel$ has been used,
so that $\Psi$ is determined by \eqref{A BVP 1}--\eqref{A BVP 4} and \eqref{A BVP 5} is
equivalent to \eqref{HW decomposition of vv}.) The significance of the quantity $\vv_\parallel$ has
previously been noted by Gavrilyuk \emph{et al.} \cite[\S 3.1]{GavrilyukKalischKhorsand15} (a study of kinematic balance laws) and Castro \& Lannes \cite{CastroLannes15}
(Hamiltonian formulations of water waves with general distributions of vorticity).

The appropriate generalisation $H(\eta)$ of the Dirichlet-Neumann operator $G(\eta)$ is identified in Section \ref{GDNO}:
one defines $H(\eta)\Phi = \curl \AA\cdot\NN|_{y=\eta}$. It is shown that
$$\mathcal{L}(\eta,\Phi) = \int_{{\mathbb R}^2} \bigg(\tfrac{1}{2}\Phi H(\eta)\Phi 
-\nabla \Phi\cdot\AA^{\!\star\!\perp}_\parallel
+\Gamma(\eta)+ \tfrac{1}{2}g\eta^2 + \sigma\big((1+|\nabla\eta|^2)^\frac{1}{2}-1\big)\bigg)$$
and that its Euler-Lagrange equations can be written as
\begin{align*}
& H(\eta)\Phi +\uu^\star\cdot\NN\big|_{y=\eta} =0, \\
& \tfrac{1}{2}|K(\eta)\Phi|^2 -\frac{(H(\eta)\Phi+K(\eta)\Phi\cdot\nabla\eta)^2}{2(1+|\nabla\eta|^2)}
-\alpha\frac{H(\eta)\Phi (H(\eta)\Phi+K(\eta)\Phi\cdot\nabla\eta)}{1+|\nabla\eta|^2} \\
& \qquad\mbox{}+K(\eta)\Phi\cdot\udl{\uu}^\star_\mathrm{h}\big|_{y=\eta}
+g\eta- \sigma\!\left(\frac{\eta_x}{(1+|\nabla\eta|^2)^\frac{1}{2}}\right)_{\!\!\!x}
-\sigma\!\left(\frac{\eta_z}{(1+|\nabla\eta|^2)^\frac{1}{2}}\right)_{\!\!\!z}= 0,
\end{align*}
where
$$K(\eta)\Phi=\nabla \Phi - \alpha \nabla^\perp \Delta^{-1} (H(\eta)\Phi).$$
(In the irrotational limit $\alpha=0$ one finds that $\curl \AA=\grad \phi$, where $\phi$ is the unique harmonic function
with $\phi_n|_{y=-h}=0$ and $\phi|_{y=\eta} = \Phi$, so that $H(\eta)\Phi = \grad \phi.\NN|_{y=\eta} = G(\eta)\Phi$,
thus recovering the Craig-Sulem-Zakharov formulation.)

The treatment of the variational principle \eqref{VP} in Section \ref{VP} consists in
computing the formal first variation $\delta {\mathcal L}(\eta,\Phi)$ of the variational functional
in terms of the infinitesimal variations $\dot{\eta}$, $\dot{\Phi}$; all
variables are supposed to be as smooth as required for the relevant calculations. The mathematics
can be made rigorous by `flattening' the variable fluid domain $D_\eta$, that is mapping it to the fixed
reference domain $D_0$ by introducing the new vertical coordinate $\tilde{y}=h(y-\eta)/(y+\eta)$ and
variable $\tilde{\AA}(x,\tilde{y},z)=\AA(x,y,z)$. The variational functional is transformed into
\begin{align*}
\mathcal{L}&(\eta,\Phi)=\int_{D_0}\!\! \bigg(\tfrac{1}{2} |\curl^{\eta} \tilde{\AA}|^2 - \tfrac{1}{2}\alpha\tilde{\AA}\cdot\curl^{\eta} \tilde{\AA}\bigg)\!\!\bigg(1+\frac{\eta}{h}\bigg)\\
&+\int_{{\mathbb R}^2} \!\!\bigg(\!- \tfrac{1}{2}\alpha \nabla \Delta^{-1}(\nabla\cdot\tilde{\AA}^{\!\perp}_\parallel)\cdot\tilde{\AA}_\parallel -\nabla \Phi\cdot\AA^{\!\star\!\perp}_\parallel
+\Gamma(\eta)+ \tfrac{1}{2}g\eta^2+ \sigma\big((1+|\nabla\eta|^2)^\frac{1}{2}-1\big)\bigg),
\end{align*}
in which $\tilde{\AA}$ is the solution of the `flattened' boundary-value problem 
\begin{align}
& \parbox{7.25cm}{$\curl^{\eta} \curl^{\eta} {\tilde{\AA}} = \alpha\, \curl^{\eta}{\tilde{\AA}}$}\mbox{in $D_0$,} \label{Intro flat 1} \\
& \parbox{7.25cm}{$\div^{\eta} {\tilde{\AA}} = 0$}\mbox{in $D_0$,} \label{Intro flat 2} \\
& \parbox{7.25cm}{${\tilde{\AA}} \wedge \jj = \mathbf{0}$}\mbox{at $\tilde{y}=-h$,} \label{Intro flat 3} \\
& \parbox{7.25cm}{${\tilde{\AA}}\cdot\NN = 0$}\mbox{at $\tilde{y}=0$,} \label{Intro flat 4} \\
& \parbox{7.25cm}{$(\curl^{\eta} {\tilde{\AA}})_\parallel = \nabla \Phi  - \alpha \nabla^\perp \Delta^{-1} (\nabla\cdot\tilde{\AA}^{\!\perp}_\parallel)
$}\mbox{at $\tilde{y}=0$} \label{Intro flat 5}
\end{align}
and the notation $\tilde{\FF}_\parallel = \tilde{\FF}_\mathrm{h}+\tilde{F}_2\nabla \eta\big|_{\tilde{y}=0}$ for vector fields
$\tilde{\FF}: D_0 \rightarrow {\mathbb R}^3$ is used; explicit expressions for
$\curl^{\eta} \tilde{\AA} := \curl \AA$, $\curl^{\eta}\curl^{\eta}\tilde{\AA}:=\curl\curl\AA$ and 
$\div^{\eta} \tilde{\AA} := \div \AA$
are given below.
This technique is used in Section \ref{FA}\ref{FA Analyticity of GDNO}, where it is shown that the nonlocal operator $H(\eta)$ depends analytically
upon $\eta$ in a sense made precise there.

The variational principle presented here is a combination of a classical result for
Beltrami flows in fixed domains (see Woltjer \cite{Woltjer58} and Laurence \& Avellaneda \cite{LaurenceAvellaneda91}) and
a suggestion for an alternative variational framework for three-dimensional irrotational
water waves by Benjamin \cite[\S 6.6]{Benjamin84}. An alternative variational principle has been
given by Lokharu \& Wahl\'{e}n \cite{LokharuWahlen19}, who use a vector potential $\AA$ within the flow as the principal
variable and consider more general parametrisations of the free surface; in the present context
their work shows that equations \eqref{Beltrami 1}--\eqref{Beltrami 5}
(with $\vv=\curl \AA$) follow from the variational principle
$$\delta\Bigg\{\!\!\int_{D_\eta}\!\!\! \bigg(\tfrac{1}{2} |\curl \AA|^2 - \tfrac{1}{2}\alpha\AA\cdot\curl \AA\bigg)
\!\!-\!\int_{{\mathbb R}^2}\!\!\!\bigg(
\frac{|\cc|^2}{2\alpha}(\sin \alpha \eta - \alpha \eta) + \tfrac{1}{2}g\eta^2 + \sigma\big((1+|\nabla\eta|^2)^\frac{1}{2}-1\big)\!\bigg)\!\!\Bigg\}\!=\!0,
$$
where the variations are taken with respect to $\eta$ and $\AA$ satisfying $\div \AA=0$,
$\AA \wedge \jj|_{y=-h}=\mathbf{0}$ and $\AA \wedge \nn|_{y=\eta}=-\AA^\star \wedge \nn|_{y=\eta}$.

\subsection{Notation and vector identities}

In this article vector fields $D_\eta \rightarrow {\mathbb R}^3$ and
${\mathbb R}^2 \rightarrow {\mathbb R}^2$ are written in respectively bold upper and lower case,
the horizontal component of $\FF=(F_1,F_2,F_3)^\mathrm{T}$ is denoted by $\FF_\mathrm{h}=(F_1,F_3)^\mathrm{T}$,
and the vector perpendicular to $\ff=(f_1,f_3)$ is denoted by
$\ff^\perp=(-f_3,f_1)$.
Evaluation at the free surface is indicated by an underscore, so that $\udl{\FF}=(F_1,F_2,F_3)^\mathrm{T}|_{y=\eta}$,
and frequent use is made of the quantity $\FF_\parallel = \udl{\FF}_\mathrm{h}+\udl{F}_2\nabla \eta$
(the horizontal component of the tangential part of $\FF$ at the free surface). The usual
three-dimensional vector operators are denoted by `$\grad$', `$\div$' and `$\curl$', while $\nabla=(\partial_x,\partial_z)^\mathrm{T}$;
the two- and three-dimensional Laplacians are both denoted by $\Delta$ (the precise meaning being clear from the context).

In Sections \ref{VP} and \ref{GDNO} we proceed formally, assuming that all functions are 
as regular as required for the relevant calculations and making frequent use of the following identities
(which are proved by explicit computation).

\begin{proposition}
The identities
\begin{itemize}
\item[(i)]
$\ff^\perp\cdot\gg^\perp = \ff\cdot\gg$, $\ff\cdot\gg^\perp = -\ff^\perp\cdot\gg$, $\ff^{\perp\perp}=-\ff$,
\item[(ii)]
$(\nabla f)^\perp = \nabla^\perp f$, $\nabla^{\perp\perp} f = - \nabla f$,
\item[(iii)]
$\ff^\perp\cdot\nabla^\perp g = \ff\cdot\nabla g$, $\ff\cdot\nabla^\perp g = - \ff^\perp\cdot\nabla g$, $\nabla\cdot\ff^\perp = -\nabla^\perp\cdot\ff$,
\item[(iv)]
$(\udl{\FF}\wedge\NN)\cdot\udl{\GG}=\FF_\parallel^\perp\cdot\GG_\parallel$,
$\udl{\curl \FF}\cdot\NN=\nabla\cdot\FF_\parallel^\perp$,
\item[(v)]
$(\grad f)_\parallel = \nabla \udl{f}$,
$(\FF_{\!\!y})_\parallel = \nabla \udl{\FF}_2+(\udl{\curl \FF})_\mathrm{h}^\perp$,
$(\udl{\curl \FF})_\mathrm{h}\cdot\FF_\parallel - \udl{\FF}\cdot\udl{\curl \FF}
= - \nabla\cdot \FF^{\!\perp}_\parallel\, \udl{\FF}_2$,
\item[(vii)]
$\displaystyle\int_{{\mathbb R}^2} \nabla f\cdot \nabla^\perp g=0, \ 
\int_{{\mathbb R}^2} \nabla f\cdot \GG = -\int_{{\mathbb R}^2} f \nabla\cdot \GG, \ 
\int_{{\mathbb R}^2} \nabla^\perp f\cdot \GG = -\int_{{\mathbb R}^2} f \nabla^\perp\cdot \GG$
\end{itemize}
are satisfied by all sufficiently regular vector fields $\FF$, $\GG:D_\eta \rightarrow {\mathbb R}^3$,
$\ff$, $\gg: {\mathbb R}^2 \rightarrow {\mathbb R}^2$
and scalar fields $f, g: {\mathbb R} \rightarrow {\mathbb R}$.
\end{proposition}

Each (sufficiently regular) vector field $\ff: {\mathbb R}^2 \rightarrow {\mathbb R}^2$ admits a unique orthogonal decomposition
$$
\ff = \nabla \Phi + \nabla^\perp \Psi,
$$
where $\Phi=\Delta^{-1} (\nabla\cdot \ff)$ and $\Psi=\Delta^{-1} (\nabla^\perp \cdot \ff)=-\Delta^{-1}(\nabla \cdot \ff^\perp)$.
Note that the projections $\ff \mapsto \nabla \Delta^{-1}(\nabla\cdot \ff)$,
$\ff \mapsto \nabla^\perp \Delta^{-1}(\nabla^\perp \cdot \ff)$ onto the `gradient part' and `orthogonal gradient part'
in the decomposition
$$\ff=\nabla \Delta^{-1}(\nabla\cdot \ff) + \nabla^\perp \Delta^{-1}(\nabla^\perp\cdot \ff)$$
are formally self-adjoint and have the property recorded in the following proposition.
\begin{proposition} \label{Inverse Laplacian proposition}
The identity 
$$\int_{{\mathbb R}^2} \nabla \Delta^{-1} (\nabla\cdot\gg^\perp)\cdot\ff - \int_{{\mathbb R}^2} \nabla \Delta^{-1} (\nabla\cdot\ff^\perp)\cdot\gg = \int_{{\mathbb R}^2} \ff\cdot\gg^\perp$$
holds for all sufficiently regular vector fields $\ff$, $\gg: {\mathbb R}^2 \rightarrow {\mathbb R}^2$.
\end{proposition}
A rigorous discussion of the Hodge-Weyl decomposition (which is used formally in 
Sections \ref{VP} and \ref{GDNO}) is given in Section \ref{FA}\ref{FA prerequisites}.

It remains to record the expressions $\curl^{\eta}\tilde{\FF}$, $\div^{\eta}\tilde{\FF}$, $\grad^\eta f$ and $\Delta^\eta f$
obtained from $\curl \FF$, $\div \FF$, $\grad \tilde{f}$ and $\Delta \tilde{f}$ by the `flattening' change of variables $\tilde{y}=h(y-\eta)/(h+\eta)$,
$\tilde{\FF}(x,\tilde{y},z)=\FF(x,y,z)$, $\tilde{f}(x,\tilde{y},z)=f(x,y,z)$; one finds that
\begin{align*}
\curl^{\eta} \tilde{\FF} &=  \curl \tilde{\FF}  -K_2^\eta(\tilde{F}_{3\tilde{y}},0,-\tilde{F}_{1\tilde{y}})+K_1^\eta(\eta_z\tilde{F}_{2\tilde{y}},\eta_x\tilde{F}_{3\tilde{y}}-\eta_{z}\tilde{F}_{1\tilde{y}},-\eta_x\tilde{F}_{2\tilde{y}}), \\
\div^{\eta} \tilde{\FF} & = \div \tilde{\FF} -K_1^\eta(\eta_x \tilde{F}_{1\tilde{y}}+\eta_z \tilde{F}_{3\tilde{y}})
-K_2^\eta \tilde{F}_{2\tilde{y}}, \\
\grad^\eta \tilde{f} & = \grad \tilde{f} - K_1^\eta \tilde{f}_y(\eta_x,0,\eta_z)-K_2^\eta \tilde{f}_y(0,1,0), \\
\Delta^\eta \tilde{f} & = \Delta \tilde{f}+K_1^\eta(\eta_x^2+\eta_z^2)(K_1^\eta\tilde{f}_{yy}+2K_3^\eta\tilde{f}_y) \\
 & \qquad\mbox{}  +K_2^\eta(K_2^\eta-2)\tilde{f}_{yy} - K_1^\eta(\eta_{xx}+\eta_{zz})\tilde{f}_y
 - 2K_1^\eta(\eta_x \tilde{f}_{xy} + \eta_z \tilde{f}_{yz}),
 \end{align*}
where $K_1^\eta=(h+\tilde{y})/(h+\eta)$, $K_2^\eta = \eta/(h+\eta)$, $K_3^\eta=1/(h+\eta)$.

\section{The variational principle} \label{VP}
In this section we verify that equations \eqref{Beltrami 1}--\eqref{Beltrami 5} (with $\vv=\curl \AA$) follow from the variational principle
\begin{align*}
&\delta\Bigg\{\int_{D_\eta}\!\! \bigg(\tfrac{1}{2} |\curl \AA|^2 - \tfrac{1}{2}\alpha\AA\cdot\curl \AA\bigg)
+\int_{{\mathbb R}^2} \!\!\bigg(\!\!- \tfrac{1}{2}\alpha \nabla \Delta^{-1}(\nabla\cdot\AA^{\!\perp}_\parallel)\cdot\AA_\parallel -\nabla \Phi\cdot\AA^{\!\star\!\perp}_\parallel\bigg) \\
& \qquad\mbox{}  +\int_{{\mathbb R}^2}\!\!\bigg(
\Gamma(\eta) + \tfrac{1}{2}g\eta^2 + \sigma\big((1+|\nabla\eta|^2)^\frac{1}{2}-1\big)\bigg)\Bigg\}=0,
\end{align*}
where
$$\Gamma(\eta) = -\tfrac{1}{2}\alpha\nabla\Delta^{-1}(\nabla\cdot\AA^{\!\star\!\perp}_\parallel)\cdot\AA^{\!\star}_\parallel+ \frac{|\cc|^2}{2\alpha}(\sin \alpha \eta - \alpha \eta),$$
$\AA$ is the unique solution of the boundary-value problem \eqref{A BVP 1}--\eqref{A BVP 5} and the
variations are taken with respect to $\eta$ and $\Phi$; equations
\eqref{Beltrami 1}--\eqref{Beltrami 3} are automatically satisfied, while \eqref{Beltrami 4} and \eqref{Beltrami 5}
are recovered from the Euler-Lagrange equations for the variational functional ${\mathcal L}(\eta,\Phi)$.
To this end note that the rules
$$\delta \udl{\FF} = \udl{\dot{\mathbf F}}+\udl{\FF_{\!\!y}} \dot{\eta}, \qquad
\delta \FF_\parallel=\dot{\FF}_\parallel
+(\FF_{\!\!y})_\parallel \dot{\eta}+\nabla\dot{\eta}\udl{\FF}_2,$$
where, as is customary, $\delta \FF$ is abbreviated to $\dot{\FF}$, imply that
\begin{align*}
&\curl \curl \dot{\AA}=\alpha\, \curl \dot{\AA} \qquad \mbox{in $D_\eta$}, \\
& \dot{\AA} \wedge \jj|_{y=-h} = \mathbf{0}, \\
&(\curl \dot{\AA})_\parallel
=-(\curl \AA_y)_\parallel \dot{\eta}-\nabla\dot{\eta}(\udl{\curl \AA})_2+\nabla \dot{\Phi} - \alpha \nabla^\perp \Delta^{-1} \big(\nabla\cdot(\delta\AA_\parallel)^\perp\big).
\end{align*}

Observe that
\begin{eqnarray}
\lefteqn{\delta \int_{D_\eta} \left(\tfrac{1}{2}|\curl \AA|^2-\tfrac{1}{2}\alpha \AA\cdot\curl \AA\right)} \nonumber \\
& = & \int_{D_\eta} \left(\curl \AA\cdot\curl\dot{\AA} -\tfrac{1}{2}\alpha \dot{\AA}\cdot\curl \AA - \tfrac{1}{2}\alpha \AA\cdot\curl \dot{\AA}\right) \nonumber \\
& & \qquad \mbox{}
+ \int_{{\mathbb R}^2} \left (\tfrac{1}{2}|\udl{\curl \AA}|^2 -\tfrac{1}{2}\alpha \AA\cdot\udl{\curl \AA}\right)\dot{\eta} \nonumber \\
& = & \int_{D_\eta} \left(\curl \curl \dot{\AA}-\alpha\, \curl \dot{\AA}\right)\cdot\AA
+ \int_{{\mathbb R}^2} (\udl{\curl \dot{\AA}}\wedge\NN)\cdot\udl{\AA} - \tfrac{1}{2}\alpha \int_{{\mathbb R}^2} (\udl{\dot{\AA}}\wedge \NN)\cdot\udl{\AA}\nonumber \\
& & \qquad \mbox{}+ \int_{{\mathbb R}^2} \left (\tfrac{1}{2}|\udl{\curl \AA}|^2 -\tfrac{1}{2}\alpha \udl{\AA}\cdot\udl{\curl \AA}\right)\dot{\eta} \nonumber \\
& = & \int_{{\mathbb R}^2} (\curl \dot{\AA})_\parallel^\perp\cdot \AA_\parallel
- \tfrac{1}{2}\alpha \int_{{\mathbb R}^2} \dot{\AA}^{\!\perp}_\parallel\cdot \AA_\parallel
+ \int_{{\mathbb R}^2} \left (\tfrac{1}{2}|\udl{\curl \AA}|^2 -\tfrac{1}{2}\alpha \udl{\AA}\cdot\udl{\curl \AA}\right)\dot{\eta} \label{Variation 1}
\end{eqnarray}
and
\begin{align}
\int_{{\mathbb R}^2} &(\curl \dot{\AA})_\parallel^\perp\cdot \AA_\parallel \nonumber \\
& = - \int_{{\mathbb R}^2} (\curl \dot{\AA})_\parallel\cdot \AA^{\!\perp}_\parallel \nonumber \\
& = \int_{{\mathbb R}^2} \big((\curl \AA_y)_\parallel \dot{\eta}+\nabla\dot{\eta}(\udl{\curl \AA})_2\big)\cdot\AA^{\!\perp}_\parallel
-\int_{{\mathbb R}^2} \nabla \dot{\Phi}\cdot\AA^{\!\perp}_\parallel \nonumber \\
& \hspace{0.5in}\mbox{}
+ \alpha \int_{{\mathbb R}^2} \nabla^\perp \Delta^{-1}(\nabla\cdot(\delta \AA_\parallel)^\perp)\cdot\AA^{\!\perp}_\parallel \nonumber \\
& = \int_{{\mathbb R}^2} \big((\curl \AA_y)_\parallel -\nabla(\udl{\curl \AA})_2\big)\cdot\dot{\eta}\AA^{\!\perp}_\parallel
-\int_{{\mathbb R}^2} (\udl{\curl \AA})_2\dot{\eta}\nabla\cdot\AA^{\!\perp}_\parallel
-\int_{{\mathbb R}^2} \nabla \dot{\Phi}\cdot\AA^{\!\perp}_\parallel \nonumber \\
& \hspace{0.5in}\mbox{}
+ \alpha \int_{{\mathbb R}^2} \nabla^\perp \Delta^{-1}(\nabla\cdot(\delta \AA_\parallel)^\perp)\cdot\AA^{\!\perp}_\parallel \nonumber \\
& = \alpha\int_{{\mathbb R}^2} (\udl{\curl \AA})_\mathrm{h}^\perp\dot{\eta}\cdot\AA^{\!\perp}_\parallel
-\int_{{\mathbb R}^2} (\udl{\curl \AA})_2\dot{\eta}\nabla\cdot\AA^{\!\perp}_\parallel
-\int_{{\mathbb R}^2} \nabla \dot{\Phi}\cdot\AA^{\!\perp}_\parallel  \nonumber \\
& \hspace{0.5in}\mbox{}
+ \alpha \int_{{\mathbb R}^2} \nabla \Delta^{-1}(\nabla\cdot(\delta \AA_\parallel)^\perp)\cdot\AA_\parallel,
\label{Variation 2}
\end{align}
where an integration by parts and the fact that
$$(\curl \AA_y)_\parallel -\nabla(\udl{\curl \AA})_2=(\udl{\curl\curl \AA})_\mathrm{h}^\perp=\alpha(\udl{\curl \AA})_\mathrm{h}^\perp$$
has been used. Combining \eqref{Variation 1}, \eqref{Variation 2} and the calculation
$$
\delta\int_{{\mathbb R}^2} \nabla \Delta^{-1}(\nabla\cdot\AA^{\!\perp}_\parallel)\cdot\AA_\parallel
= \int_{{\mathbb R}^2} \nabla \Delta^{-1}(\nabla\cdot(\delta \AA_\parallel)^\perp)\cdot\AA_\parallel
+ \int_{{\mathbb R}^2} \nabla \Delta^{-1}(\nabla\cdot\AA^{\!\perp}_\parallel)\cdot\delta\AA_\parallel
$$
yields
\begin{eqnarray}
\lefteqn{\delta \left\{\int_{D_\eta} \left(\tfrac{1}{2}|\curl \AA|^2-\tfrac{1}{2}\alpha \AA\cdot\curl \AA\right)
-\tfrac{1}{2}\alpha \int_{{\mathbb R}^2} \nabla \Delta^{-1}(\nabla\cdot\AA^{\!\perp}_\parallel)\cdot\AA_\parallel\right\}}\nonumber \\
& = & \alpha\int_{{\mathbb R}^2} (\udl{\curl \AA})_\mathrm{h}\dot{\eta}\cdot\AA_\parallel
-\int_{{\mathbb R}^2} (\udl{\curl \AA})_2\dot{\eta}\nabla\cdot\AA^{\!\perp}_\parallel
-\int_{{\mathbb R}^2} \nabla \dot{\Phi}\cdot\AA^{\!\perp}_\parallel 
-\tfrac{1}{2}\alpha \int_{{\mathbb R}^2} \dot{\AA}_\parallel^\perp\cdot\AA_\parallel \nonumber \\
& & \qquad\mbox{}+ \int_{{\mathbb R}^2} \left (\tfrac{1}{2}|\udl{\curl \AA}|^2 -\tfrac{1}{2}\alpha \udl{\AA}\cdot\udl{\curl \AA}\right)\dot{\eta}-\tfrac{1}{2}\alpha \int_{{\mathbb R}^2} \AA^{\!\perp}_\parallel\cdot\delta\AA_\parallel,
\label{Variation 3}
\end{eqnarray}
where Proposition \ref{Inverse Laplacian proposition} has also been used.
Repeating the argument leading to \eqref{Variation 2}, one finds that
\begin{align}
\int_{{\mathbb R}^2} \AA^{\!\perp}_\parallel\cdot\delta\AA_\parallel 
& = \int_{{\mathbb R}^2} \AA^{\!\perp}_\parallel\cdot \big(
\dot{\AA}_\parallel+(\AA_y)_\parallel \dot{\eta}+\nabla\dot{\eta}\udl{\AA}_2
\big) \nonumber \\
& = \int_{{\mathbb R}^2} \AA^{\!\perp}_\parallel\cdot\dot{\AA}_\parallel
+\int_{{\mathbb R}^2}\big((\AA_y)_\parallel -\nabla(\udl{\AA}_2)\big)\dot{\eta}\cdot\AA^{\!\perp}_\parallel
-\int_{{\mathbb R}^2} \udl{\AA}_2\dot{\eta} \nabla\cdot\AA^{\!\perp}_\parallel \nonumber \\
& =- \int_{{\mathbb R}^2} \AA_\parallel\cdot\dot{\AA}_\parallel^\perp
+\int_{{\mathbb R}^2}(\udl{\curl \AA})_\mathrm{h}\dot{\eta}\cdot\AA_\parallel
-\int_{{\mathbb R}^2} \udl{\AA}_2\dot{\eta} \nabla\cdot\AA^{\!\perp}_\parallel, \label{Variation 4}
\end{align}
and it follows from \eqref{Variation 3}, \eqref{Variation 4} and the calculation
$$
(\udl{\curl \AA})_\mathrm{h}\cdot\AA_\parallel - \udl{\AA}\cdot\udl{\curl \AA}
= - \nabla\cdot \AA^{\!\perp}_\parallel\, \udl{\AA}_2
$$
that\pagebreak
\begin{align*}
&\delta \left\{\int_{D_\eta} \left(\tfrac{1}{2}|\curl \AA|^2-\tfrac{1}{2}\alpha \AA\cdot\curl \AA\right)
-\tfrac{1}{2}\alpha \int_{{\mathbb R}^2} \nabla \Delta^{-1}(\nabla\cdot\AA^{\!\perp}_\parallel)\cdot\AA_\parallel\right\} \\
& \qquad =-\int_{{\mathbb R}^2} (\udl{\curl \AA})_2\dot{\eta}\nabla\cdot\AA^{\!\perp}_\parallel
+\int_{{\mathbb R}^2} \dot{\Phi}\, \nabla\cdot\AA^{\!\perp}_\parallel + \int_{{\mathbb R}^2} \tfrac{1}{2}|\udl{\curl \AA}|^2\dot{\eta}.
\end{align*}

Finally, note that
\begin{align*}
\delta\int_{{\mathbb R}^2}
\big( -\tfrac{1}{2}\alpha&\nabla\Delta^{-1}(\nabla\cdot\AA^{\!\star\!\perp}_\parallel)\cdot\AA^{\!\star}_\parallel\big)\\
&=
- \tfrac{1}{2}\alpha\int_{{\mathbb R}^2}\big(
\nabla\Delta^{-1}(\nabla\cdot\AA^{\!\star\!\perp}_\parallel)\cdot\dot{\AA}^{\!\star}_\parallel
+\nabla\Delta^{-1}(\nabla\cdot\AA^{\!\star}_\parallel)\cdot\dot{\AA}^{\!\star\!\perp}_\parallel\big)\\
&=
- \tfrac{1}{2}\alpha\int_{{\mathbb R}^2}\big(
\nabla^\perp\Delta^{-1}(\nabla\cdot\AA^{\!\star\!\perp}_\parallel)\cdot\dot{\AA}^{\!\star\!\perp}_\parallel
+\nabla\Delta^{-1}(\nabla\cdot\AA^{\!\star}_\parallel)\cdot\dot{\AA}^{\!\star\!\perp}_\parallel\big)\\
&=
 \tfrac{1}{2}\alpha\int_{{\mathbb R}^2}\big(
\nabla^\perp\Delta^{-1}(\nabla\cdot\AA^{\!\star\!\perp}_\parallel)
+\nabla\Delta^{-1}(\nabla\cdot\AA^{\!\star}_\parallel)\big)\cdot\uu^\star_\parallel\dot{\eta} \\
&=
 \tfrac{1}{2}\alpha\int_{{\mathbb R}^2}\big(\AA^{\!\star}_\parallel
 +2\nabla^\perp\Delta^{-1}(\nabla\cdot\AA^{\!\star\!\perp}_\parallel)\big)\cdot\uu^\star_\parallel\dot{\eta}, \\
 \\
 \delta\int_{{\mathbb R}^2}
 \frac{|\cc|^2}{2\alpha}(\sin \alpha \eta - \alpha \eta)
&= -|\cc|^2\!\!\int_{{\mathbb R}^2}\sin^2(\tfrac{1}{2}\alpha \eta)\dot{\eta} =
-\tfrac{1}{2}\alpha\int_{{\mathbb R}^2} \AA^{\!\star}_\parallel.\uu^\star_\parallel \dot{\eta}
\end{align*}
and
\begin{align*}
\delta \int_{{\mathbb R}^2}\!\! (-\nabla \Phi\cdot\AA^{\!\star\!\perp}_\parallel)
& = \int_{{\mathbb R}^2} \!\!\big(\dot{\Phi}\nabla\cdot\AA^{\!\star\!\perp}_\parallel-
\nabla\Phi\cdot \dot{\AA}^{\!\star\perp}_\parallel \big)\\
& = \int_{{\mathbb R}^2}\!\! \big(\dot{\Phi}\nabla\cdot\AA^{\!\star\!\perp}_\parallel
+\nabla\Phi\cdot \uu^{\star}_\parallel \dot{\eta} \big)\\
& = \int_{{\mathbb R}^2} \!\!\big(\dot{\Phi}\nabla\cdot\AA^{\!\star\!\perp}_\parallel
+\big((\curl \AA)_\parallel + \alpha \nabla^\perp \Delta^{-1} (\nabla\cdot\AA^{\!\perp}_\parallel)\big)\cdot \uu^{\star}_\parallel \dot{\eta} \big)
\\& = \int_{{\mathbb R}^2}\!\! \big(\dot{\Phi}\nabla\cdot\AA^{\!\star\!\perp}_\parallel
+\big((\udl{\curl \AA})_\mathrm{h}
+(\udl{\curl \AA})_2\nabla\eta
 + \alpha \nabla^\perp \Delta^{-1} (\nabla\cdot\AA^{\!\perp}_\parallel)\big)\cdot \uu^{\star}_\parallel \dot{\eta} \big)
\end{align*}
because
$\dot{\AA}^{\!\star\!\perp}_\parallel = - \uu^\star_\parallel\dot{\eta}$,
and evidently
$$\delta \int_{{\mathbb R}^2}\bigg( \tfrac{1}{2}g\eta^2 + \sigma\big((1+|\nabla\eta|^2)^\frac{1}{2}-1\big)\bigg)
\!=\!
\int_{{\mathbb R}^2}\bigg(
g\eta- \sigma\!\left(\frac{\eta_x}{(1+|\nabla\eta|^2)^\frac{1}{2}}\right)_{\!\!\!x}-\sigma\!\left(\frac{\eta_z}{(1+|\nabla\eta|^2)^\frac{1}{2}}\right)_{\!\!\!z}
\bigg)\dot{\eta}.$$
The Euler-Lagrange equations for ${\mathcal L}(\eta,\Phi)$ are therefore
\begin{align}
&\nabla\cdot \AA^{\!\perp}_\parallel+\nabla\cdot \AA^{\!\star\!\perp}_\parallel =0, \label{EL 1}\\
& (\udl{\curl \AA})_2 \big(-\nabla\cdot \AA^{\!\perp}_\parallel+\nabla\eta\cdot\uu^\star_\parallel\big)
+\alpha \nabla^\perp\Delta^{-1}\big(\nabla\cdot\AA^{\!\perp}_\parallel+\nabla\cdot \AA^{\!\star\!\perp}_\parallel\big)\cdot\uu^\star_\parallel\nonumber \\
& \qquad\mbox{}
+ \tfrac{1}{2}|\udl{\curl \AA}|^2+(\udl{\curl \AA})_\mathrm{h}.\uu^\star_\parallel
+g\eta- \sigma\!\left(\frac{\eta_x}{(1+|\nabla\eta|^2)^\frac{1}{2}}\right)_{\!\!\!x}\!\!-\sigma\!\left(\frac{\eta_z}{(1+|\nabla\eta|^2)^\frac{1}{2}}\right)_{\!\!\!z}=0, \label{EL 2}
\end{align}
which are equivalent to equations \eqref{Beltrami 4}, \eqref{Beltrami 5} because
$\nabla\cdot \AA^{\!\star\!\perp}_\parallel=
-\nabla\eta.\uu^\star_\parallel=\udl{\uu}^\star\cdot\NN$,
$\udl{\curl \AA}_\mathrm{h}.\uu^\star_\parallel
=\udl{\curl \AA}.\udl{\uu}^\star$ and
$\nabla\cdot \AA^{\!\perp}_\parallel = \udl{\curl \AA}\cdot\NN$.

\section{A nonlocal operator} \label{GDNO}

In this section we express the variational functional ${\mathcal L}(\eta,\Phi)$ and its Euler-Lagrange equations
in terms of a nonlocal operator $H(\eta)$ defined as follows: for fixed $\Phi$, let $\AA$ denote the unique solution of
\eqref{A BVP 1}--\eqref{A BVP 5} and define
\begin{equation}
H(\eta)\Phi = \nabla\cdot\AA^{\!\perp}_\parallel. \label{Defn of H}
\end{equation}

\begin{lemma} \label{Intermediate H result}
The formula
\begin{align*}
\int_{{\mathbb R}^2} \Phi_1 H(\eta) \Phi_2 & = \int_{D_\eta} \left(\curl \BB\cdot\curl \CC-\tfrac{1}{2}\alpha \BB\cdot\curl \CC-\tfrac{1}{2}\alpha \CC\cdot\curl \BB\right) \nonumber \\
& \qquad\mbox{}- \tfrac{1}{2}\alpha\int_{{\mathbb R}^2} \nabla \Delta^{-1}(\nabla\cdot\BB^{\!\perp}_\parallel)\cdot\CC_\parallel
-\tfrac{1}{2}\alpha\int_{{\mathbb R}^2} \nabla \Delta^{-1}(\nabla\cdot\CC^{\!\perp}_\parallel)\cdot\BB_\parallel
\end{align*}
holds for all $\Phi_1$, $\Phi_2$, where $\BB$ and $\CC$ denote the unique solutions of \eqref{A BVP 1}--\eqref{A BVP 5}
with respectively $\Phi=\Phi_1$ and $\Phi=\Phi_2$ (so that
$H(\eta)\Phi_1 = \nabla\cdot\BB^{\!\perp}_\parallel$, $H(\eta)\Phi_2 = \nabla\cdot\CC^{\!\perp}_\parallel$).

In particular, $H(\eta)$ is formally self-adjoint, that is
$$\int_{{\mathbb R}^2} \Phi_1 H(\eta) \Phi_2 = \int_{{\mathbb R}^2} \Phi_2 H(\eta) \Phi_1$$
for all $\Phi_1$, $\Phi_2$.
\end{lemma}
\begin{proof}
Note that
\begin{align*}
& \int_{D_\eta} \left(\curl \BB\cdot\curl \CC-\tfrac{1}{2}\alpha \BB\cdot\curl \CC-\tfrac{1}{2}\alpha \CC\cdot\curl \BB\right) \\
& = \int_{D_\eta} \left(\curl \curl \BB-\alpha\, \curl \BB\right)\cdot\CC
+\int_{{\mathbb R}^2} (\udl{\curl \BB}\wedge\NN)\cdot\udl{\CC} - \tfrac{1}{2}\alpha \int_{{\mathbb R}^2} (\udl{\BB}\wedge \NN)\cdot\udl{\CC} \\
& = -\int_{{\mathbb R}^2} (\curl \BB)_\parallel\cdot\CC^{\!\perp}_\parallel
-\tfrac{1}{2}\alpha \int_{{\mathbb R}^2} \BB^{\!\perp}_\parallel\cdot\CC_\parallel \\
& = -\int_{{\mathbb R}^2} \nabla \Phi_1\cdot\CC^{\!\perp}_\parallel
+ \alpha\int_{{\mathbb R}^2} \nabla^\perp \Delta^{-1}(\nabla\cdot\BB^{\!\perp}_\parallel)\cdot\CC^{\!\perp}_\parallel
-\tfrac{1}{2}\alpha \int_{{\mathbb R}^2} \BB^{\!\perp}_\parallel\cdot\CC_\parallel \\
& = \int_{{\mathbb R}^2} \Phi_1 \nabla\cdot\CC^{\!\perp}_\parallel 
+ \alpha\int_{{\mathbb R}^2} \nabla \Delta^{-1}(\nabla\cdot\BB^{\!\perp}_\parallel)\cdot\CC_\parallel
-\tfrac{1}{2}\alpha \int_{{\mathbb R}^2} \BB^{\!\perp}_\parallel\cdot\CC_\parallel \\
& = \int_{{\mathbb R}^2} \Phi_1 H(\eta) \Phi_2
+ \alpha\int_{{\mathbb R}^2} \nabla \Delta^{-1}(\nabla\cdot\BB^{\!\perp}_\parallel)\cdot\CC_\parallel 
-\tfrac{1}{2}\alpha \int_{{\mathbb R}^2} \BB^{\!\perp}_\parallel\cdot\CC_\parallel,
\end{align*}
and thus
\begin{align}
& \int_{D_\eta} \left(\curl \BB\cdot\curl \CC-\tfrac{1}{2}\alpha \BB\cdot\curl \CC-\tfrac{1}{2}\alpha \CC\cdot\curl \BB\right) \nonumber \\
& \qquad\mbox{}- \tfrac{1}{2}\alpha\int_{{\mathbb R}^2} \nabla^\perp \Delta^{-1}(\nabla\cdot\BB^{\!\perp}_\parallel)\cdot\CC^{\!\perp}_\parallel
-\tfrac{1}{2}\alpha\int_{{\mathbb R}^2} \nabla^\perp \Delta^{-1}(\nabla\cdot\CC^{\!\perp}_\parallel)\cdot\BB^{\!\perp}_\parallel
\label{Symmetric formula} \\
& = \int_{{\mathbb R}^2} \Phi_1 H(\eta) \Phi_2 \nonumber \\
& \qquad\mbox{}
+ \tfrac{1}{2}\alpha\int_{{\mathbb R}^2} \nabla \Delta^{-1}(\nabla\cdot\BB^{\!\perp}_\parallel)\cdot\CC_\parallel
-\tfrac{1}{2} \alpha\int_{{\mathbb R}^2} \nabla \Delta^{-1}(\nabla\cdot\CC^{\!\perp}_\parallel)\cdot\BB_\parallel 
-\tfrac{1}{2}\alpha \int_{{\mathbb R}^2} \BB^{\!\perp}_\parallel\cdot\CC_\parallel  \nonumber \\
& = \int_{{\mathbb R}^2} \Phi_1 H(\eta) \Phi_2, \nonumber
\end{align}
where the last line follows by Proposition \ref{Inverse Laplacian proposition}.
\end{proof}

\begin{theorem}
The variational functional $\mathcal{L}(\eta,\Phi)$ and its Euler-Lagrange equations may be written as respectively
$$\mathcal{L}(\eta,\Phi) = \int_{{\mathbb R}^2} \bigg(\tfrac{1}{2}\Phi H(\eta)\Phi 
-\nabla \Phi\cdot\AA^{\!\star\!\perp}_\parallel
+\Gamma(\eta)+ \tfrac{1}{2}g\eta^2 + \sigma\big((1+|\nabla\eta|^2)^\frac{1}{2}-1\big)\bigg)$$
and\pagebreak
\begin{align*}
& H(\eta)\Phi + \udl{\uu}^\star\cdot\NN=0, \\
& \tfrac{1}{2}|K(\eta)\Phi|^2 -\frac{(H(\eta)\Phi+K(\eta)\Phi\cdot\nabla\eta)^2}{2(1+|\nabla\eta|^2)}
-\alpha\frac{H(\eta)\Phi (H(\eta)\Phi+K(\eta)\Phi\cdot\nabla\eta)}{1+|\nabla\eta|^2} \\
& \qquad\mbox{}+K(\eta)\Phi\cdot\udl{\uu}^\star_\mathrm{h}
+g\eta- \sigma\!\left(\frac{\eta_x}{(1+|\nabla\eta|^2)^\frac{1}{2}}\right)_{\!\!\!x}
-\sigma\!\left(\frac{\eta_z}{(1+|\nabla\eta|^2)^\frac{1}{2}}\right)_{\!\!\!z}= 0,
\end{align*}
where
$K(\eta)\Phi=\nabla \Phi - \alpha \nabla^\perp \Delta^{-1} (H(\eta)\Phi)$.
\end{theorem}
\proof Using Lemma \ref{Intermediate H result} with $\Phi_1=\Phi_2=\Phi$, one finds that
$$
\int_{D_\eta} \left(\tfrac{1}{2}|\curl \AA|^2-\tfrac{1}{2}\alpha \AA\cdot\curl \AA\right)
- \tfrac{1}{2}\alpha\int_{{\mathbb R}^2} \nabla \Delta^{-1}(\nabla\cdot\AA^{\!\perp}_\parallel)\cdot\AA_\parallel 
= \tfrac{1}{2}\int_{{\mathbb R}^2} \Phi H(\eta) \Phi,
$$
where $\AA$ is the unique solution of \eqref{A BVP 1}--\eqref{A BVP 5}, and the first result follows from this formula
and the definition \eqref{Defn of L} of ${\mathcal L}(\eta,\Phi)$.
The result for the Euler-Lagrange equations is obtained from \eqref{EL 1}, \eqref{EL 2}
and the identities
$$|\udl{\vv}|^2 = |\vv_\parallel|^2 + \frac{(\udl{\vv}\cdot\NN)^2 - (\vv_\parallel\cdot\nabla\eta)^2}{1+|\nabla\eta|^2}, \qquad
\udl{\vv}_2 = \frac{\udl{\vv}\cdot\NN+\vv_\parallel\cdot\nabla\eta}{1+|\nabla\eta|^2}
$$
with $\vv=\curl \AA$, so that
$$\udl{\vv}\cdot\NN = \nabla\cdot\AA^{\!\!\perp}_\parallel = H(\eta)\Phi, \qquad
\vv_\parallel = \nabla \Phi - \alpha \nabla^\perp \Delta^{-1}(H(\eta)\Phi).\eqno{\Box}$$

\begin{remark}
Note that $\Delta^{-1} (H(\eta)\Phi)$ is well defined because $H(\eta)\Phi = \nabla\cdot\AA^{\!\perp}_\parallel$.
\end{remark}

\section{Functional-analytic aspects} \label{FA}

\subsection{Hodge-Weyl decomposition} \label{FA prerequisites}

Let us work in the Sobolev spaces
$$H^s({\mathbb R}^2)
=
\left\{u \in {\mathcal S}^\prime({\mathbb R}^2):
\| u \|_s^2 := \int_{{\mathbb R}^2} \left( 1 + |\kk|^2 \right)^{s} |\hat u(\kk)|^2 \,\mathrm{d}\kk < \infty\right\}, \qquad s \in {\mathbb R},
$$
where $\hat{u}=\mathcal{F}[u]$ denotes the Fourier transform of $u$, and the Beppo-Levi
spaces
$$\dot{H}^s({\mathbb R}^2):=\{u \in L^2_\mathrm{loc}({\mathbb R}^2): \|u\|_{\dot{H}^s({\mathbb R}^2)}
:=\|\nabla u\|_{H^{s-1}({\mathbb R}^2)} < \infty\}, \qquad s \geq 0.$$
For each $\ff \in L^2({\mathbb R}^2)^2$ there exist unique functions $\Phi$, $\Psi \in \dot{H}^1({\mathbb R}^2)$
such that
$$
\ff = \nabla \Phi + \nabla^\perp \Psi;
$$
this (obviously orthogonal) decomposition of $u$ is its \emph{Hodge-Weyl decomposition}.
The functions $\Phi, \Psi$ are characterised as the weak solutions
of the equations 
$\Delta \Phi = \nabla \cdot \ff$, $\Delta\Psi=\nabla^\perp\cdot \ff$, that is
$\Phi$ and $\Psi$ are the unique functions in $\dot{H}^1({\mathbb R})$ such that
$$\int_{{\mathbb R}^2} \nabla \Phi \cdot \nabla \chi = \int_{{\mathbb R}^2} \ff\cdot\nabla \chi,
\qquad
\int_{{\mathbb R}^2} \nabla \Psi \cdot \nabla \chi = \int_{{\mathbb R}^2} \ff\cdot\nabla^\perp \chi$$
for all $\chi \in \dot{H}^1({\mathbb R}^2)$, and one accordingly
writes $\Phi=\Delta^{-1} (\nabla\cdot \ff)$, $\Psi=\Delta^{-1} (\nabla^\perp \cdot \ff)$.

\begin{proposition}
For each $s \geq 0$ the formulae
$\ff \mapsto \Delta^{-1}(\nabla \cdot \ff)$ and $\ff \mapsto \Delta^{-1}(\nabla^\perp \cdot \ff)$
define continuous linear mappings $H^s({\mathbb R}^2)^2\rightarrow\dot{H}^{s+1}({\mathbb R}^2)$.
\end{proposition}

\subsection{Well-posedness of the defining boundary-value problem} \label{FA for BVP}

In this section we suppose that $\eta$ is a fixed function in $W^{2,\infty}({\mathbb R}^2)$ with $\inf \eta > -h$ and study the
boundary-value problem \eqref{A BVP 1}--\eqref{A BVP 5} using the standard Sobolev spaces $L^2(D_\eta)^3$
and $H^1(D_\eta)^3$ together with the closed subspaces
\begin{align*}
{\mathcal X}_\eta & = \{\FF \in H^1(D_\eta)^3\!: \FF \wedge \jj |_{y=-h}= \mathbf{0},\, \udl{\FF}\cdot\nn = 0\}, \\
{\mathcal Y}_\eta & = \{\FF \in H^1(D_\eta)^3\!: \FF\cdot\jj|_{y=-h}=0,\, \FF_\parallel^\perp = {\bf 0}\}
\end{align*}
of $H^1(D_\eta)^3$. The following proposition gives an alternative description of ${\mathcal X}_\eta$
and ${\mathcal Y}_\eta$ (see Castro \& Lannes \cite[Lemma 3.3]{CastroLannes15} for the result for ${\mathcal X}_\eta$;
the result for ${\mathcal Y}_\eta$ is established in an analogous fashion).

\begin{proposition} \label{Alternative for X, Y}
The spaces ${\mathcal X}_\eta$ and ${\mathcal Y}_\eta$ coincide with respectively
\begin{align*}
& \{\FF \in L^2(D_\eta)^3\!: \curl \FF\in L^2(D_\eta)^3,\, \div \FF \in L^2(D_\eta),\,  \FF \wedge \jj |_{y=-h}= \mathbf{0},\, \udl{\FF}\cdot\nn = 0\}, \\
& \{\FF \in L^2(D_\eta)^3\!: \curl \FF\in L^2(D_\eta)^3,\, \div \FF \in L^2(D_\eta),\, \FF\cdot\jj|_{y=-h}=0,\, \FF_\parallel^\perp = {\bf 0}\},
\end{align*}
and the function
$\FF \mapsto (\|\curl \FF\|_{L^2(D_\eta)^3}^2+\|\div \FF\|_{L^2(D_\eta)}^2)^\frac{1}{2}$ is equivalent to their usual norm.
\end{proposition}

A \emph{weak solution} of \eqref{A BVP 1}--\eqref{A BVP 5} is a function
$\AA \in {\mathcal X}_\eta$ such that
\begin{equation}
\int_{D_\eta} (\curl \AA\cdot\curl \BB -\alpha\, \curl \AA \cdot\BB + \div \AA\, \div \BB)
-\alpha\int_{{\mathbb R}^2} \nabla \Delta^{-1}( \nabla \cdot \AA^{\!\perp}_\parallel) \cdot \BB_\parallel
=
\int_{{\mathbb R}^2} \nabla^\perp \Phi\cdot\BB_\parallel
\label{Weak solution}
\end{equation}
for all $\BB \in {\mathcal X}_\eta$,
while a \emph{strong solution} has the additional regularity requirement that $\AA \in H^2(D_\eta)^3$, is solenoidal
and satisfies
\eqref{A BVP 1} in $L^2(D_\eta)^3$ and \eqref{A BVP 5} in $H^\frac{1}{2}({\mathbb R}^2)^2$.
The existence of weak and strong solutions is established in Lemmata \ref{Existence of w solution} and
\ref{Existence of s solution} below; both are preceded by auxiliary results (Propositions \ref{Properties of spaces}, \ref{BVP proposition 1}
and \ref{More about spaces}) necessary for their proof.

\begin{proposition} \label{Properties of spaces}
\hspace{1cm}\begin{itemize}
\item[(i)]
The function $\FF \mapsto \udl{\FF}$ is a continuous linear mapping
$H^1(D_\eta)^3 \rightarrow H^\frac{1}{2}({\mathbb R}^2)^3$ with continuous right inverse
$H^\frac{1}{2}({\mathbb R}^2)^3 \rightarrow \{\FF \in H^1(D_\eta)^3: \FF|_{y=-h}={\bf 0}\}$.
\item[(ii)]
The mapping $\FF \mapsto \FF_\parallel^\perp$ defined on ${\mathcal D}(\overline{D_\eta})^3$
extends to a continuous linear mapping $\{\FF \in L^2(D_\eta)^3: \curl \FF \in L^2(D_\eta)^3\} \rightarrow H^{-\frac{1}{2}}({\mathbb R}^2)^2$,
where the former space is equipped with the norm $\FF \mapsto (\|\FF\|_{L^2(D_\eta)^3}^2+\|\curl \FF\|_{L^2(D_\eta)^3}^2)^\frac{1}{2}$.
\end{itemize}
\end{proposition}
\begin{proof}
Assertion (i) follows from the corresponding result for $\tilde{\FF}: D_0 \rightarrow {\mathbb R}^3$, the fact that
$\udl{\FF}=\tilde{\FF}|_{\tilde{y}=0}$ and the estimates
$\|\tilde{\FF}\|_{H^1(D_0)} \lesssim \|\FF\|_{H^1(D_\eta)}$,
$\|\FF\|_{H^1(D_\eta)} \lesssim \|\tilde{\FF}\|_{H^1(D_0)}$, while
(ii) is obtained by a standard argument from the identity
$$\int_{D_\eta} (\FF\cdot\curl\GG-\curl\FF\cdot\GG) = \int_{{\mathbb R}^2} \FF_\parallel^\perp\cdot\GG_\parallel,
\qquad \FF \in {\mathcal D}(\overline{D_\eta}),\ \GG \in H^1(D_\eta), \GG|_{y=-h}=0$$
(see Dautray \& Lions \cite[p.\ 207]{DautrayLions3}).
\end{proof}

The proof of the next proposition has been given by Lannes \cite[ch.\ 2]{Lannes}.

\begin{proposition} \label{BVP proposition 1}
 The boundary-value problem
\begin{align*}
\Delta u &=G && \hspace{-1.5cm}\mbox{in $D_\eta$,} \\
\partial_n u & = 0 && \hspace{-1.5cm}\mbox{at $y=\eta$,} \\
u &=0 &&\hspace{-1.5cm} \mbox{at $y=-h$}
\end{align*}
has a unique solution $u \in H^2(D_\eta)$ for each $G \in L^2(D_\eta)$.
\end{proposition}

\begin{lemma} \label{Existence of w solution}
For all sufficiently small values of $|\alpha|$ the
boundary-value problem \eqref{A BVP 1}--\eqref{A BVP 5} admits a unique weak solution
for each $\Phi \in \dot{H}^\frac{1}{2}({\mathbb R}^2)$. The weak solution is solenoidal and
satisfies \eqref{A BVP 1} in the sense of distributions and \eqref{A BVP 5} in $H^{-\frac{1}{2}}({\mathbb R}^2)^2$.
\end{lemma}
\begin{proof}
The estimates
\begin{align*}
\left| \int_{D_\eta} \curl \AA \cdot\BB \right| & \lesssim \|\AA\|_{H^1(D_\eta)^3} \|\BB\|_{H^1(D_\eta)^3}, \\
\left| \int_{{\mathbb R}^2} \nabla \Delta^{-1}( \nabla \cdot \AA^{\!\perp}_\parallel) \cdot \BB_\parallel \right|
& \lesssim \|\AA_\parallel\|_0 \|\BB_\parallel\|_0
\lesssim \|\udl{\AA}\|_{\frac{1}{2}} \|\udl{\BB}\|_{\frac{1}{2}}
\lesssim   \|\AA\|_{H^1(D_\eta)^3} \|\BB\|_{H^1(D_\eta)^3}
\end{align*}
and Proposition \ref{Alternative for X, Y} imply that for sufficiently small values of $|\alpha|$
the left-hand side of \eqref{Weak solution}
is a continuous, coercive, bilinear form $\mathcal{X}_\eta \times \mathcal{X}_\eta \rightarrow {\mathbb R}$,
while the estimate
$$
\left| \int_{{\mathbb R}^2} \nabla^\perp \Phi \cdot \BB_\parallel \right| \lesssim
\|\nabla^\perp \Phi\|_{-\frac{1}{2}}\|\udl{\BB}\|_{\frac{1}{2}} \lesssim
 \|\Phi\|_{\dot{H}^\frac{1}{2}({\mathbb R}^2)} \|\BB\|_{H^1(D_\eta)^3}
$$
shows that its right-hand side is a continuous, bilinear form $\dot{H}^\frac{1}{2}({\mathbb R}^2)^2 \times {\mathcal X}_\eta
\rightarrow {\mathbb R}$ (note that Proposition \ref{Properties of spaces}(i) has been used in both steps).
The existence of a unique solution $\AA \in {\mathcal X}_\eta$
now follows from the Lax-Milgram lemma.

Let $\phi_\AA \in H^2(D_\eta)$ be the unique function satisfying
$\Delta \phi_\AA = \div \AA$ in $D_\eta$ with boundary conditions
$\partial_n \phi_\AA|_{y=\eta} =0$, $\phi_\AA|_{y=-h}=0$ (see Proposition \ref{BVP proposition 1}).
It follows that $\BB=\grad \phi_\AA \in {\mathcal X}_\eta$ and hence
$$
\int_{D_\eta} \big(-\alpha\, \curl \AA \cdot \grad \phi_\AA + (\div \AA)^2\big)
-\alpha\int_{{\mathbb R}^2} \nabla \Delta^{-1}( \nabla \cdot \AA^{\!\perp}_\parallel) \cdot \nabla (\udl{\phi_\AA}) \\
=
0
$$
(because $\BB_\parallel = \nabla ( \udl{\phi_\AA})$, which is orthogonal
to $\nabla^\perp \Phi$). Since
$$
\alpha\!\int_{D_\eta} \!\!\curl \AA \cdot \grad \phi_\AA
= \alpha\!\int_{{\mathbb R}^2} \udl{\curl \AA} \cdot \NN\, \udl{\phi_\AA}
=\alpha\!\int_{{\mathbb R}^2} \!\nabla \cdot \AA^{\!\perp}_\parallel\, \udl{\phi_\AA}
= -\alpha\! \int_{{\mathbb R}^2} \nabla \Delta^{-1}(\nabla \cdot \AA^{\!\perp}_\parallel)\cdot \nabla (\udl{\phi_\AA}),
$$
one concludes that $\div \AA=0$.

Choosing $\BB \in {\mathcal D}(D_\eta)^3$, one finds that $\AA$ solves \eqref{A BVP 1}
in the sense of distributions and hence that $\curl \curl \AA \in L^2(D_\eta)^3$. It follows that
$(\curl \AA)_\parallel^\perp \in H^{-\frac{1}{2}}({\mathbb R}^2)^2$ (Proposition \ref{Properties of spaces}(ii)) and
$$
\int_{D_\eta} (\curl \curl \AA -\alpha\, \curl \AA) \cdot\BB
+\int_{{\mathbb R}^2} \left((\curl \AA)^\perp_\parallel -\nabla^\perp \Phi
-\alpha\nabla \Delta^{-1}( \nabla \cdot \AA^{\!\perp}_\parallel)\right) \cdot \BB_\parallel
=0.
$$
One concludes that \eqref{A BVP 5} holds in $H^{-\frac{1}{2}}({\mathbb R}^2)^2$.
\end{proof}

\begin{proposition} \label{More about spaces}
\hspace{1cm}
\begin{itemize}
\item[(i)]
The spaces
$$\{\FF \!\in\! L^2(D_\eta)^3\!: \curl \FF\!\in\! L^2(D_\eta)^3,\, \div \FF \!\in\! L^2(D_\eta),\, \FF\wedge \jj\big|_{y=-h} \!\in\! H^\frac{1}{2}({\mathbb R}^2)^3,\, \udl{\FF}.\NN \!\in\! H^\frac{1}{2}({\mathbb R}^2)\}$$
and
$$\{\FF \!\in\! L^2(D_\eta)^3\!: \curl \FF\!\in\! L^2(D_\eta)^3,\, \div \FF \!\in\! L^2(D_\eta),\, \FF\cdot\jj\big|_{y=-h} \!\in\! H^\frac{1}{2}({\mathbb R}^2),\, \FF_\parallel^\perp \!\in\! H^\frac{1}{2}({\mathbb R}^2)^2\}$$
coincide with $H^1(D_\eta)^3$.
\item[(ii)]
The space
$$\{\FF \in L^2(D_\eta)^3\!: \curl \FF\in H^1(D_\eta)^3,\, \div \FF \in H^1(D_\eta),\, \FF \wedge \jj \big|_{y=-h}= \mathbf{0},\, \udl{\FF}\cdot\nn = 0\}$$
coincides with $\{\FF \in H^2(D_\eta)^3: \FF \wedge \jj \big|_{y=-h}= \mathbf{0},\, \udl{\FF}\cdot\nn = 0\}$.
\end{itemize}
\end{proposition}
\begin{proof}
(i) Comparing the Sobolev-Slobodeckij norms for the two spaces (see Adams \cite[\S7.48]{Adams}) shows that
$\udl{\FF}.\nn \in H^\frac{1}{2}(S_\eta)$ if and only if $\udl{\FF}.\NN \in H^\frac{1}{2}({\mathbb R}^2)$ and
$\udl{\FF} \wedge \nn \in H^\frac{1}{2}(S_\eta)^3$ if and only if $\udl{\FF} \wedge \NN \in H^\frac{1}{2}({\mathbb R}^2)^3$;
since $(\FF \wedge \NN)_\mathrm{h} = \FF_\parallel^\perp$ and $(\FF \wedge \NN)_2 = \FF_\parallel^\perp\cdot\nabla \eta$
it follows that $\udl{\FF} \wedge \nn \in H^\frac{1}{2}(S_\eta)^3$ if and only if $\FF_\parallel^\perp \in H^\frac{1}{2}({\mathbb R}^2)$.

The given spaces obviously contain $H^1(D_\eta)^3$; our task to establish the reverse inclusions.
Suppose that $\FF \in L^2(D_\eta)^3$ satisfies $\curl \FF\in L^2(D_\eta)^3$,
$\div \FF \in L^2(D_\eta)$,
$\FF\wedge\jj\big|_{y=-h} \in H^\frac{1}{2}({\mathbb R}^2)^2$ and
$\udl{\FF}.\NN \in H^\frac{1}{2}({\mathbb R}^2)$, so that
$\udl{\FF}.\nn \in H^\frac{1}{2}(S_\eta)$.
Letting $\phi \in H^2(D_\eta)$ be a function with
$$\partial_n \phi\big|_{y=\eta} = \udl{\FF}.\nn$$
and $\GG \in H^1(D_\eta)^3$ be a function with
$$\udl{\GG} = {\bf 0}, \qquad \GG\big|_{y=-h}=-(\FF \wedge \jj) \wedge \jj + (\grad \phi \wedge \jj) \wedge \jj\big|_{y=-h},$$
we find that $\HH=\FF-\grad\phi - \GG$ satisfies $\HH \in L^2(D_\eta)^3$, $\curl \HH\in L^2(D_\eta)^3$,
$\div \HH \in L^2(D_\eta)$, $\udl{\HH}.\nn=0$ and $\HH \wedge \jj\big|_{y=-h} = {\bf 0}$.
Proposition \ref{Alternative for X, Y} asserts that
$\HH \in \mathcal{X}_\eta$ and hence $\FF = \HH + \grad \phi + \GG \in H^1(D_\eta)$.

Similarly, suppose 
that $\FF \in L^2(D_\eta)^3$ satisfies $\curl \FF\in L^2(D_\eta)^3$,
$\div \FF \in L^2(D_\eta)$, $\FF\cdot\jj\big|_{y=-h} \in H^\frac{1}{2}({\mathbb R}^2)$ and
$\FF_\parallel^\perp \in H^\frac{1}{2}({\mathbb R}^2)^2$, so that
$\udl{\FF} \wedge \nn \in H^\frac{1}{2}(S_\eta)^3$. Letting $\phi \in H^2(D_\eta)$ be a function with
$$\partial_n \phi\big|_{y=-h} = \FF.\jj\big|_{y=-h}$$
and $\GG \in H^1(D_\eta)^3$ be a function with
$$\GG\big|_{y=-h} = {\bf 0}, \qquad \udl{\GG}=-(\udl{\FF} \wedge \nn) \wedge \nn + (\udl{\grad \phi} \wedge \nn) \wedge \nn,$$
we find that $\HH=\FF-\grad\phi - \GG$ satisfies $\HH \in L^2(D_\eta)^3$, $\curl \HH\in L^2(D_\eta)^3$,
$\div \HH \in L^2(D_\eta)$, $\HH.\jj\big|_{y=-h}=0$ and $\udl{\HH} \wedge \nn = {\bf 0}$,
so that $\HH_\parallel^\perp={\bf 0}$. Proposition \ref{Alternative for X, Y} asserts that
$\HH \in \mathcal{Y}_\eta$ and hence $\FF = \HH + \grad \phi + \GG \in H^1(D_\eta)$.

(ii) This result follows by applying the results in part (i) to the derivatives of $\FF$.
\end{proof}

\begin{lemma} \label{Existence of s solution}
Suppose that $\Phi \in \dot{H}^\frac{3}{2}({\mathbb R}^2)$. Any weak solution $\AA$ of
\eqref{A BVP 1}--\eqref{A BVP 5} is in fact a strong solution.
\end{lemma}

\begin{proof}
Recall that $\curl \curl \AA \in L^2(D_\eta)^3$ and
$$
(\curl \AA)^\perp_\parallel =\nabla^\perp \Phi
+\alpha\nabla \Delta^{-1}( \nabla \cdot \AA^{\!\perp}_\parallel)
$$
holds in $H^{-\frac{1}{2}}({\mathbb R}^2)^2$; hence
$(\curl \AA)_\parallel^\perp \in H^{\frac{1}{2}}({\mathbb R}^2)^2$
(because the right-hand side of this equation belongs to $H^{\frac{1}{2}}({\mathbb R}^2)^2$).
Since $0=\div \curl \AA \in L^2(D_\eta)$
and $\curl \AA\cdot\jj|_{y=-h} =0$ it follows that $\curl \AA \in H^1(D_\eta)^3$ (Proposition
\ref{More about spaces}(i)), and furthermore
$\curl \AA \in H^1(D_\eta)^3$, $0=\div \AA \in H^1(D_\eta)$ with
$\AA \wedge \jj |_{y=-h}= \mathbf{0}$, $\udl{\AA}\cdot\nn = 0$ imply that $\AA \in H^2(D_\eta)^3$
(Proposition \ref{More about spaces}(ii)). Finally note that \eqref{A BVP 1} holds in $L^2(D_\eta)^3$ because
it holds in the sense of distributions and $\AA \in H^2(D_\eta)^3$.
\end{proof}

We conclude this section with the following alternative characterisation of a strong solution to
\eqref{A BVP 1}--\eqref{A BVP 5}.

\begin{proposition} \label{Equivalent BVPs}
Suppose that $\Phi \in \dot{H}^\frac{3}{2}({\mathbb R}^2)$.
The strong solutions of the boundary-value problems \eqref{A BVP 1}--\eqref{A BVP 5} and
\begin{align}
& \parbox{7.25cm}{$-\Delta \AA = \alpha\, \curl\AA$}\mbox{in $D_\eta$,} \label{alt BVP 1}\\
& \parbox{7.25cm}{$\AA \wedge \jj = \mathbf{0}$}\mbox{at $y=-h$,} \label{alt BVP 2}\\
& \parbox{7.25cm}{$A_{2y} = 0$}\mbox{at $y=-h$,} \label{alt BVP 3}\\
& \parbox{7.25cm}{$\AA\cdot\nn = 0$}\mbox{at $y=\eta$,} \label{alt BVP 4}\\
& \parbox{7.25cm}{$(\curl \AA)_\parallel = \nabla \Phi - \alpha \nabla^\perp \Delta^{-1} (\nabla\cdot\AA^{\!\perp}_\parallel)
$}\mbox{at $y=\eta$} \label{alt BVP 5}
\end{align}
coincide (so that in particular \eqref{alt BVP 1}--\eqref{alt BVP 5} has a unique strong solution).
\end{proposition}

\begin{proof}
Suppose that $\AA \in H^2(D_\eta)^3$ is a strong solution of \eqref{A BVP 1}--\eqref{A BVP 5}, so that $\AA$ satisfies
\eqref{alt BVP 1} (due to the identity $\curl \curl \AA = -\Delta \AA + \grad \div \AA$) and
$\div \AA\big|_{y=-h}=0$.
Since $\AA_\mathrm{h}\big|_{y=-h}=0$ and hence $\nabla \cdot \AA_\mathrm{h}\big|_{y=-h}=0$, we conclude that
$\mbox{$A_{2y}\big|_{y=-h} = \div \AA - \nabla \cdot \AA_\mathrm{h}\big|_{y=-h}=0$.}$

The above argument shows that any strong solution $\AA  \in H^2(D_\eta)^3$ of \eqref{alt BVP 1}--\eqref{alt BVP 5} satisfies
$\div \AA\big|_{y=-h}=0$; it remains to show that in fact $\div \AA=0$ in $D_\eta$.
Writing $-\Delta \AA = \curl \curl \AA - \grad \div \AA$, taking the scalar product of \eqref{alt BVP 1} with a
function $\BB \in {\mathcal X}_\eta$, and integrating by parts using the integral identities
\begin{align*}
\int_{D_\eta} \curl \curl \AA\cdot\BB &= \int_{D_\eta} \curl \AA\cdot\curl \BB - \int_{{\mathbb R}^2} (\curl \AA)^{\!\perp}_\parallel\cdot\BB_\parallel, \\
\int_{D_\eta} \grad \div \AA\cdot\BB &= -\int_{D_\eta} \div \AA \div \BB
\end{align*}
(where we have used $\BB \wedge \jj\big|_{y=-h}={\mathbf 0}$,
$\underline{\BB}\cdot\nn=0$ and $\div \AA\big|_{y=-h}=0$), we find that
$\AA$ satisfies \eqref{Weak solution}.
It follows that $\AA$ is a weak
solution of \eqref{A BVP 1}--\eqref{A BVP 5}, so that in particular $\div \AA=0$.
\end{proof}

\subsection{Analyticity of the operator $H(\eta)$} \label{FA Analyticity of GDNO}

In this section we improve the result of Lemmata \ref{Existence of w solution} and \ref{Existence of s solution}
by quantifying the restriction that $|\alpha|$ is small and showing that improved regularity of $\eta$ and $\Phi$
yields improved regularity of $\AA$; we use these results to deduce that $H(\eta)$ depends analytically upon
$\eta$ (see Corollary \ref{H is anal} below for a precise statement of this result). The starting point is the
`flattened' version \eqref{Intro flat 1}--\eqref{Intro flat 5} of the boundary-value problem \eqref{A BVP 1}--\eqref{A BVP 5},
which according to Proposition \ref{Equivalent BVPs} is equivalent to the `flattened' version of the boundary-value problem
\eqref{alt BVP 1}--\eqref{alt BVP 5}, that is
\begin{align}
& \parbox{8.25cm}{$-\Delta {\tilde{\AA}} - \alpha\, \curl{\tilde{\AA}}=\HH^\eta(\tilde{\AA})$}\mbox{in $D_0$,} \label{alt flat 1} \\
& \parbox{8.25cm}{${\tilde{\AA}} \wedge \jj = \mathbf{0}$}\mbox{at $y=-h$,} \label{alt flat 2} \\
& \parbox{8.25cm}{$\tilde{A}_{2y} = 0$}\mbox{at $y=-h$,} \label{alt flat 3} \\
& \parbox{8.25cm}{$\udl{\tilde{\AA}}\cdot\jj = g^\eta(\tilde{\AA})$}\mbox{at $y=0$,} \label{alt flat 4} \\
& \parbox{8.25cm}{$(\udl{\curl {\tilde{\AA}}})_\mathrm{h} + \alpha \nabla^\perp \Delta^{-1} (\nabla\cdot\udl{\tilde{\AA}}^{\!\perp}_\mathrm{h})= \hh^\eta(\tilde{\AA})+\nabla \Phi 
$}\mbox{at $y=0$}, \label{alt flat 5}
\end{align}
where
\begin{align*}
\HH^\eta(\tilde{\AA}) &= \Delta^\eta \tilde{\AA} + \alpha \curl^\eta \tilde{\AA} - \Delta \tilde{\AA} - \alpha \curl \tilde{\AA}, \\
g^\eta(\tilde{A}) &=\nabla \eta \cdot \udl{\tilde{\AA}}_\mathrm{h}, \\
\hh^\eta(\tilde{A}) &= -(\udl{\curl^{\eta}{\tilde{\AA}}})_\mathrm{h}+(\udl{\curl \tilde{\AA}})_\mathrm{h} 
-\nabla \eta (\udl{\curl^{\eta}\tilde{\AA}})_2
-\alpha\nabla^\perp\!\Delta^{-1}(\nabla \cdot(\nabla\eta^\perp \udl{\tilde{A}}_2);
\end{align*}
the definition \eqref{Defn of H} of $H$ is
accordingly replaced by
\begin{equation}
H(\eta)\Phi = \nabla\cdot\udl{\tilde{\AA}}^{\!\perp}_\mathrm{h} + \nabla \cdot(\nabla\eta^\perp \udl{\tilde{A}}_2).
\label{Flattened defn of H}
\end{equation}
(With a slight abuse of notation $\tilde{y}$ has been replaced by $y$ for notational simplicity
and the underscore now denotes evaluation at $y=0$).

The discussion of the boundary-value problem \eqref{alt flat 1}--\eqref{alt flat 5} begins with the corresponding
inhomogeneous linear problem.\vspace{-0.182\baselineskip}

\begin{proposition} \label{Solve the inhomo linear BVP}
Suppose that $s \geq 2$ and $\alpha^\star < \frac{\pi}{2h}$. The boundary-value problem
\begin{align*}
& \parbox{8.25cm}{$-\Delta \tilde{\AA} - \alpha\, \curl{\tilde{\AA}}=\HH$}\mbox{in $D_0$,} \\
& \parbox{8.25cm}{${\tilde{\AA}} \wedge \jj = \mathbf{0}$}\mbox{at $y=-h$,} \\
& \parbox{8.25cm}{$\tilde{A}_{2y} = 0$}\mbox{at $y=-h$,} \\
& \parbox{8.25cm}{$\udl{\tilde{\AA}}\cdot\jj = g$}\mbox{at $y=0$,} \\
& \parbox{8.25cm}{$(\udl{\curl {\tilde{\AA}}})_\mathrm{h} + \alpha \nabla^\perp \Delta^{-1} (\nabla\cdot\udl{\tilde{\AA}}^{\!\perp}_\mathrm{h})= \hh$}\mbox{at $y=0$}
\end{align*}
has a unique solution $\tilde{\AA} \in H^s(D_0)^3$ for each $g \in H^{s-\frac{1}{2}}({\mathbb R}^2)$ and
$\HH \in H^{s-2}(D_0)^3$, $\hh \in  H^{s-\frac{3}{2}}({\mathbb R}^2)^2$. The solution operator defines a
mapping $H^{s-\frac{1}{2}}({\mathbb R}^2) \times H^{s-2}(D_0)^3 \times H^{s-\frac{3}{2}}({\mathbb R}^2)^2
\rightarrow H^s(D_0)^3$ which is bounded uniformly over $|\alpha| \in [0,\alpha^\star]$.
\end{proposition}

\begin{proof}
The solution to this boundary-value problem is
$$\tilde{\mathbf A} = {\mathcal F}^{-1}\left[\int_{-h}^0 G(y,\zeta) \hat{\mathbf H}(\zeta) \dzeta
-G(y,0)\begin{pmatrix}\hat{h}_1-\i k_1 \hat{g} \\ 0 \\ \hat{h}_3-\i k_3 \hat{g} \end{pmatrix}
-G_\zeta(y,0)\begin{pmatrix} 0 \\ \hat{g} \\ 0 \end{pmatrix}\right],$$
where $\hat{u}=\mathcal{F}[u]$ denotes the Fourier transform of $u$ with respect to $(x,z)$ (with independent variable $\kk=(k_1,k_3)$) and
the Green's matrix $G$ is given by
$$G(y,\zeta) = \left\{\begin{array}{cc}U(y+h)\overline{C}^\mathrm{T}\overline{W(\zeta)}^\mathrm{T}, & -h \leq y \leq \zeta \leq 0,\\
W(y)C\overline{U(\zeta+h)}^\mathrm{T}, & -h \leq \zeta \leq y \leq 0;
\end{array}
\right.$$
explicit formulae for the matrices $U(y)$, $C$ and $W(y)$ are stated in Appendix \ref{Green's}.
We show by induction that the formulae 
$${\mathcal G}_1(\HH) = {\mathcal F}^{-1}\left[\int_{-h}^0 G(y,\zeta) \hat{\mathbf H}(\zeta) \dzeta\right],$$
$$
{\mathcal G}_2(\hh) = {\mathcal F}^{-1}\left[G(y,0)\hat{\mathbf h}(\zeta) \dzeta\right],
\qquad
{\mathcal G}_3(\hh) = {\mathcal F}^{-1}\left[G_\zeta(y,0)\hat{\mathbf h}(\zeta) \dzeta\right]$$
define linear mappings ${\mathcal G}_1: H^m(D_0)^3 \rightarrow H^{m+2}(D_0)^3$,
${\mathcal G}_2: H^{m+\frac{1}{2}}({\mathbb R}^2)^3 \rightarrow H^{m+2}(D_0)^3$,
${\mathcal G}_3: H^{m+\frac{3}{2}}({\mathbb R}^2)^3 \rightarrow H^{m+2}(D_0)^3$ which are bounded
uniformly over $|\alpha| \in [0,\alpha^\star]$ for $m=0,1,2,\ldots$
and hence for $m \in [0,\infty)$ by interpolation.

The estimates given in Appendix \ref{Green's} (\eqref{Appendix est 1}--\eqref{Appendix est 5}) imply
the existence of $R>0$ such that
$$|G(y,\zeta)|,\ |G_y(y,\zeta)| \lesssim 1$$
for $|\kk| \leq R$ and
$$|G(y,\zeta)| \lesssim \frac{1}{|\kk|}\e^{-|\kk||y-\zeta|}, \qquad |G_y(y,\zeta)| \lesssim \e^{-|\kk||y-\zeta|}$$
for $|\kk| \geq R$, uniformly over $y, \zeta \in [-h,0]$ and $|\alpha| \in [0,\alpha^\star]$.
Using these inequalities, one finds by straightforward calculations that
$$\|{\mathcal G}_1(\HH)\|_1,\|\partial_x{\mathcal G}_1(\HH)\|_1,
\|\partial_z{\mathcal G}_1(\HH)\|_1 \lesssim \|\HH\|_0,$$
and it follows from the equation
\begin{equation}
\partial_y^2 {\mathcal G}_1(\HH) = -\partial_x^2 {\mathcal G}_1(\HH) - \partial_z^2 {\mathcal G}_1(\HH) - \alpha \curl {\mathcal G}_1(\HH) - \HH \label{Green interior} \\
\end{equation}
that
$$\|\partial_{yy}{\mathcal G}_1(\HH)\|_0 \lesssim \|\HH\|_0.$$
This argument establishes the result for ${\mathcal G}_1$ for $m=0$.

Suppose that $\|{\mathcal G}_1(\HH)\|_{m+2} \lesssim \|\HH\|_s$ for some $m \in {\mathbb N}_0$,
so that
$$\|{\mathcal G}_1(\HH)\|_{m+2} \lesssim \|\HH\|_m \lesssim \|\HH\|_{m+1}$$
and obviously
\begin{align*}
& \| \partial_x {\mathcal G}_1(\HH)\|_{m+2} \lesssim \|{\mathcal G}_1(\partial_x \HH)\|_{m+2}
\lesssim \|\partial_x \HH\|_s \lesssim \|\HH\|_{m+1}, \\&
\| \partial_z {\mathcal G}_1(\HH)\|_{m+2} \lesssim \|{\mathcal G}_1(\partial_z \HH)\|_{m+2}
\lesssim \|\partial_z \HH\|_s \lesssim \|\HH\|_{m+1}.
\end{align*}
Furthermore, it follows from \eqref{Green interior} that
\begin{align*}
\|\partial_y^{m+3} {\mathcal G}_1(\HH) \|_0
& \lesssim 
\||D|\partial_y^{m+1} {\mathcal G}_1(|D|\HH)\|_0
+ \|\partial_y^{m+1}{\mathcal G}_1(\HH) \|_1
+\|\partial_y^{m+1} \HH\|_0 \\
& \lesssim \|{\mathcal G}_1(|D|\HH)\|_{m+2}+\|{\mathcal G}_1(\HH) \|_{m+2}
+\|\partial_y^{m+1} \HH\|_0 \\
& \lesssim \||D|\HH\|_{m} + \|\HH\|_{m} + \|\HH\|_{m+1}\\
& \lesssim \|\HH\|_{m+1},
\end{align*}
where $|D|={\mathcal F}^{-1}[|\kk| {\mathcal F}[\cdot]]$. One concludes that $\|{\mathcal G}_1(\HH)\|_{m+3} \lesssim \|\HH\|_{m+1}$.

\enlargethispage{0.2cm}The estimates for ${\mathcal G}_2$ and ${\mathcal G}_3$ are obtained in a similar fashion.
\end{proof}\pagebreak

In view of the previous result we henceforth fix $\alpha^\star < \frac{\pi}{2h}$ and assume
that $|\alpha| \in [0,\alpha^\star]$; all results hold uniformly over these values of $\alpha$.

\begin{theorem}
Suppose that $s \geq 2$.
There exists an open neighbourhood $V$ of the origin in $H^{s+\frac{1}{2}}({\mathbb R}^2)$ such that
the boundary-value problem \eqref{alt flat 1}--\eqref{alt flat 5}
has a unique solution $\tilde{\AA}=\tilde{\AA}(\eta,\Phi)$ in $H^s(D_0)^3$ for each $\eta \in V$ and $\Phi \in \dot{H}^{s-\frac{1}{2}}({\mathbb R})$.
Furthermore $\tilde{\AA}(\eta,\Phi)$ depends analytically upon $\eta$ and $\Phi$ (and linearly upon $\Phi$).
\end{theorem}
\begin{proof}
Using the estimates
\begin{align*}
\|f_1(\eta)\Delta \eta \tilde{f}\|_{H^{s-2}(D_0)} & \lesssim \|f(\eta)\|_s \|\Delta \eta\|_{s-\frac{3}{2}}\|\tilde{f}\|_{H^{s-1}(D_0)}, \\
\|f_2(\eta,\nabla \eta)\tilde{f}\|_{H^{s-2}(D_0)} & \lesssim \|f(\eta,\nabla \eta)\|_{s-\frac{1}{2}}\|\tilde{f}\|_{H^{s-2}(D_0)}, \\
\|f_2(\eta,\nabla \eta)\tilde{f}\|_{H^{s-1}(D_0)} & \lesssim \|f(\eta,\nabla \eta)\|_{s-\frac{1}{2}}\|\tilde{f}\|_{H^{s-1}(D_0)}, \\
\|\nabla \eta \udl{\tilde f}\|_{s-\frac{1}{2}} & \lesssim \|\nabla \eta\|_{s-\frac{1}{2}}\|\tilde{f}\|_{H^s(D_0)}, \\
\|\nabla \eta \udl{\tilde f}\|_{s-\frac{3}{2}} & \lesssim \|\nabla \eta\|_{s-\frac{1}{2}}\|\tilde{f}\|_{H^{s-1}(D_0)},
\end{align*}
where $f_1$ and $f_2$ are analytic mappings of a neighbourhood of the origin in respectively
${\mathbb R}$ and ${\mathbb R}^3$ into ${\mathbb R}$,
one finds that the mappings $(\eta,\tilde{\AA}) \mapsto g^\eta(\tilde{\AA})$,
$(\eta,\tilde{\AA}) \mapsto \HH^\eta(\tilde{\AA})$, $(\eta,\tilde{\AA}) \mapsto \hh^\eta(\tilde{\AA})$,
 are analytic in $V \times H^s(D_0)^3$,
where $V$ is an open neighbourhood of the origin in $H^{s+\frac{1}{2}}({\mathbb R}^2)$,
their target spaces being respectively $H^{s-\frac{1}{2}}({\mathbb R}^2)$, $H^{s-2}(D_0)^3$ and $H^{s-\frac{3}{2}}({\mathbb R}^2)^2$. (The above estimates are obtained from standard embedding theorems --- see in particular H\"{o}rmander
\cite[Theorem 8.3.1]{Hoermander}, noting that this theorem remains true when ${\mathbb R}^3$ is replaced by $D_0$).
It follows that the formula
$${\mathcal H}(\tilde{\AA},\eta,\Phi)=\begin{pmatrix} -\Delta {\tilde{\AA}} - \alpha\, \curl{\tilde{\AA}}-\HH^\eta(\tilde{\AA}) \\
\udl{\tilde{\AA}}\cdot\jj - g^\eta(\tilde{\AA}) \\
(\udl{\curl {\tilde{\AA}}})_\mathrm{h} + \alpha \nabla^\perp \Delta^{-1} (\nabla\cdot\udl{\tilde{\AA}}^{\!\perp}_\mathrm{h}) - \hh^\eta(\tilde{\AA})-\nabla \Phi 
\end{pmatrix},$$
defines an analytic mapping
$${\mathcal H}: S \times V \times \dot{H}^{s-\frac{1}{2}}({\mathbb R}) \rightarrow H^{s-2}(D_0)^3 \times H^{s-\frac{1}{2}}({\mathbb R}^2) \times
H^{s-\frac{3}{2}}({\mathbb R}^2)^2,$$
where $S=\{\tilde{\AA} \in H^s(D_0)^3: \tilde{\AA} \wedge \jj\big|_{y=-h}={\mathbf 0}, \tilde{A}_{2y}=0\}$.

Furthermore, ${\mathcal H}({\mathbf 0},0,0)=({\mathbf 0},0,{\mathbf 0})$, and the calculation
$$\mathrm{d}_1{\mathcal H}[{\mathbf 0},0,0](\tilde{\AA})=\begin{pmatrix}
\curl \curl {\tilde{\AA}} - \alpha\, \curl{\tilde{\AA}} \\
\udl{\tilde{\AA}}\cdot\jj \\
(\udl{\curl {\tilde{\AA}}})_\mathrm{h} + \alpha \nabla^\perp \Delta^{-1} (\nabla\cdot\udl{\tilde{\AA}}^{\!\perp}_\mathrm{h})
\end{pmatrix}$$
and Proposition \ref{Solve the inhomo linear BVP} show that
$$\mathrm{d}_1{\mathcal H}[{\mathbf 0},0,0]: S \times H^{s+\frac{1}{2}}({\mathbb R}^2) \times \dot{H}^{s-\frac{1}{2}}({\mathbb R}) \rightarrow H^{s-2}(D_0)^3 \times H^{s-\frac{1}{2}}({\mathbb R}^2) \times
H^{s-\frac{3}{2}}({\mathbb R}^2)^2$$
is an isomorphism.
The analytic implicit-function theorem (Buffoni \& Toland \cite[Theorem 4.5.3]{BuffoniToland})
asserts the existence of an open neighbourhoods $W$ and $U$ of the origin in respectively
$H^{s+\frac{1}{2}}({\mathbb R}^2) \times \dot{H}^{s-\frac{1}{2}}({\mathbb R})$ and
$H^s(D_0)^3$ such that
the equation
$${\mathcal H}(\tilde{\AA},\eta,\Phi)=({\mathbf 0},0,{\mathbf 0})$$
and hence the boundary-value problem \eqref{alt flat 1}--\eqref{alt flat 5}
has a unique solution $\tilde{\AA}_0=\tilde{\AA}_0(\eta,\Phi)$ in $U$ for each $(\eta,\Phi) \in W$; furthermore
$\tilde{\AA}_0(\eta,\Phi)$ depends
analytically upon $\eta$ and $\Phi$. Since $\tilde{\AA}_0$ depends linearly upon $\Phi$ one can without loss of generality take
$W=V \times \dot{H}^{s-\frac{1}{2}}({\mathbb R})$, and clearly $U=H^s(D_0)^3$
(with $\Phi=0$ the construction yields a unique solution in a neighbourhood of the origin in $H^s(D_0)^3$,
which is evidently the zero solution).
\end{proof}

\begin{corollary} Suppose that $\eta \in H^{m+\frac{1}{2}}({\mathbb R}^2) \cap W^{m,\infty}({\mathbb R}^2)$
for some $m \in \{2,3,\ldots\}$. The strong solution $\AA$ to \eqref{alt BVP 1}--\eqref{alt BVP 5} lies in $H^m(D_0)$.
\end{corollary}

\begin{corollary} \label{H is anal}
Suppose that $s \geq 2$. There exists an open neighbourhood $V$ of the origin in $H^{s+\frac{1}{2}}({\mathbb R}^2)$ such that the formula
\eqref{Flattened defn of H} defines an analytic mapping
$H(\cdot): V \rightarrow {\mathcal L}(\dot{H}^{s-\frac{1}{2}}({\mathbb R}),H^{s-\frac{3}{2}}({\mathbb R}))$.
\end{corollary}

\appendix
\section{The Green's matrix}\label{Green's}

The entries of the matrices $U(y)=(u_{mn}(y))$ and $W(y)=(w_{mn}(y))$ are
\begin{align*}
u_{11}(y) &= \frac{\i k_1}{|\kk|}\sinh |\kk|y, \\
u_{12}(y) &= -\frac{k_1}{\alpha}[c(y)-\cosh |\kk|y]+k_3\left[s_1(y)-\frac{\sinh |\kk|y}{|\kk|}\right]+\frac{k_3}{|\kk|}\sinh |\kk|y, \\
u_{13}(y) &= -\frac{\i k_3}{\alpha}[c(y)-\cosh |\kk|y]-\frac{\i k_1}{2|\kk|}\sinh |\kk|y
+ \frac{\i k_1}{\alpha^2}\left[ s_2(y)-|\kk|\sinh |\kk|y + \frac{\alpha^2}{2}y \cosh |\kk|y\right], \\
u_{21}(y) & = \cosh |\kk|y, \\
u_{22}(y) &= \frac{\i |\kk|^2}{\alpha}\left[s_1(y)-\frac{\sinh |\kk|y}{|\kk|}\right], \\
u_{23}(y) & = \frac{|\kk|^2}{\alpha^2}\left[c(y) - \cosh |\kk|y + \frac{\alpha^2}{2|\kk|}y \sinh |\kk|y \right], \\
u_{31}(y) & = \frac{\i k_3}{|\kk|}\sinh |\kk|y, \\
u_{32}(y) & = -\frac{k_3}{\alpha}[c(y)-\cosh |\kk|y] - k_1\left[s_1(y)-\frac{\sinh |\kk|y}{|\kk|}\right] + \frac{k_1}{|\kk|}\sinh |\kk|y, \\
u_{33}(y) &= \frac{\i k_1}{\alpha}[c(y)-\cosh |\kk|y]-\frac{\i k_3}{2|\kk|}\sinh |\kk|y
+ \frac{\i k_3}{\alpha^2}\left[ s_2(y)-|\kk|\sinh |\kk|y + \frac{\alpha^2}{2}y \cosh |\kk|y\right]
\end{align*}
and
\begin{align*}
w_{11}(y) &= \frac{\i k_1}{|\kk|}\cosh |\kk|y, \\
w_{12}(y) & = -\frac{\i k_3}{|\kk|}[c(y)-\cosh |\kk|y] - \frac{\i k_3}{|\kk|}\cosh |\kk|y + \frac{\i k_1}{\alpha |\kk|}[s_2(y)-|\kk|\sinh |\kk|y], \\
w_{13}(y) &= \frac{k_1}{2|\kk|}\cosh |\kk|y + \frac{|\kk| k_3}{\alpha}\left[s_1(y)-\frac{\sinh |\kk|y}{|\kk|}\right]
-\frac{|\kk| k_1}{\alpha^2}\left[c(y)-\cosh |\kk|y + \frac{\alpha^2}{2|\kk|}y\sinh |\kk|y\right], \\
w_{21}(y) &= \sinh |\kk|y, \\
w_{22}(y) &= \frac{|\kk|}{\alpha}[c(y)-\cosh |\kk|y], \\
w_{23}(y) &= \frac{\i |\kk|^3}{\alpha^2}\left[s_1(y)-\frac{\sinh |\kk|y}{|\kk|}+\frac{\alpha^2}{2|\kk|^2}y\cosh |\kk|y\right], \\
w_{31}(y) &= \frac{\i k_3}{|\kk|}\cosh |\kk|y, \\
w_{32}(y) &= \frac{\i k_1}{|\kk|}[c(y)-\cosh |\kk|y] + \frac{\i k_1}{|\kk|}\cosh |\kk|y + \frac{\i k_3}{\alpha |\kk|}[s_2(y)-|\kk|\sinh |\kk|y], \\
w_{33}(y) &= \frac{k_3}{2|\kk|}\cosh |\kk|y - \frac{|\kk| k_1}{\alpha}\left[s_1(y)-\frac{\sinh |\kk|y}{|\kk|}\right]
-\frac{|\kk| k_3}{\alpha^2}\left[c(y)-\cosh |\kk|y + \frac{\alpha^2}{2|\kk|}y\sinh |\kk|y\right],
\end{align*}
where
$$c(y) = \left\{\begin{array}{ll}\cos (\alpha^2-|\kk|^2)^{\frac{1}{2}} y, & |\kk| \leq \alpha, \\[2mm]
\cosh (|\kk|^2-\alpha^2)^{\frac{1}{2}} y, & |\kk| \geq \alpha, 
\end{array}\right.$$
$$
s_1(y) = \left\{\begin{array}{ll}\dfrac{\sin (\alpha^2-|\kk|^2)^{\frac{1}{2}}y}{(\alpha^2-|\kk|^2)^{\frac{1}{2}}}, & |\kk| \leq \alpha, \\[2mm]
\dfrac{\sinh (|\kk|^2-\alpha^2)^{\frac{1}{2}}y}{(|\kk|^2-\alpha^2)^{\frac{1}{2}}}, & |\kk| \geq \alpha,
\end{array}\right.\qquad
s_2(y) = (|\kk|^2-\alpha^2)s_1(y).
$$

Fix $\alpha^\star>0$. Straightforward estimates show that
\begin{align*}
|\partial_y^i\big(c(y)-\cosh |\kk|y\big)| & \lesssim \alpha |\kk|^{i-1}\e^{|\kk||y|}\\
\left|\partial_y^i\left(c(y) - \cosh |\kk|y + \frac{\alpha^2}{2|\kk|}y \sinh |\kk|y\right)\right| & \lesssim \alpha^3 |\kk|^{i-2}\e^{|\kk||y|}\\
\left|\partial_y^i\left(s_1(y)-\frac{\sinh |\kk|y}{|\kk|}\right)\right| & \lesssim \alpha |\kk|^{i-2}\e^{|\kk||y|} \\
\left|\partial_y^i\left(s_1(y)-\frac{\sinh |\kk|y}{|\kk|}+\frac{\alpha^2}{2|\kk|^2}y\cosh |\kk|y\right)\right| & \lesssim \alpha^3|\kk|^{i-3}\e^{|\kk||y|} \\
|\partial_y^i\big(s_2(y)-|\kk|\sinh |\kk|y\big)| & \lesssim \alpha |\kk|^{i}\e^{|\kk||y|}\\
\left|\partial_y^i\left(s_2(y)-|\kk|\sinh |\kk|y + \frac{\alpha^2}{2}y \cosh |\kk|y\right)\right| &\lesssim \alpha^3 |\kk|^{i-1}\e^{|\kk||y|}, \qquad i=0,1
\end{align*}
(uniformly over $|y| \in [0,h]$ and $|\kk| \geq \max(1,\sqrt{2}\alpha)$); it follows that
\begin{equation}
|\partial_y^i u_{mn}(y)|, |\partial_y^i w_{mn}(y)| \lesssim |\kk|^i \e^{|\kk||y|}, \qquad i=0,1 \label{Appendix est 1}
\end{equation}
(uniformly over $|\alpha| \in [0,\alpha^\star]$, $|y| \in [0,h]$ and $|\kk| \geq \max(1,\sqrt{2}\alpha)$).
Noting that
$$c(y) \rightarrow \cos\alpha y, \qquad s_1(y) \rightarrow  \frac{\sin\alpha y}{\alpha}, \qquad s_2(y) \rightarrow -\alpha \sin \alpha y$$
and
$$\cosh |\kk| y \rightarrow 1, \qquad \frac{\sinh |\kk|y}{|\kk|} \rightarrow y, \qquad |\kk| \sinh |\kk| \rightarrow 0$$
as $|\kk| \rightarrow 0$ (uniformly over $|y| \in [0,h]$),
we find that
\begin{equation}
\frac{1}{|\kk|}|u_{mn}(y)| \lesssim 1, \qquad |w_{mn}(y)| \lesssim 1\label{Appendix est 2}
\end{equation}
as $|\kk| \rightarrow 0$ (uniformly over $|y| \in [0,h]$ and $|\alpha| \in [0,\alpha^\star]$). Similarly, since $c_y(y)=s_2(y)$,
$s_{1y}(y) = c(y)$ and $s_{2y}(y) = (k^2-a^2)c(y)$, we find that
\begin{equation}
|\partial_y u_{mn}(y)| \lesssim 1, \qquad |\partial_y w_{mn}(y)| \lesssim 1 \label{Appendix est 3}
\end{equation}
as $|\kk| \rightarrow 0$ (uniformly over $|y| \in [0,h]$ and $|\alpha| \in [0,\alpha^\star]$).

The entries of the matrix $C=(c_{mn})$ are
\begin{align*}
c_{11} &= -\frac{1}{2|\kk|}\sech |\kk|h + \frac{|\kk|}{\alpha^2}\left[s-\sech |\kk|h +\frac{h\alpha^2}{2|\kk|}\sech |\kk|h \tanh |\kk|h\right]\\
& \qquad\mbox{}-\frac{2|\kk|}{\alpha^2}\tanh^2 |\kk|h [s-\sech |\kk|h]  \\
& \qquad\mbox{}+\frac{2}{\alpha^2}\sech |\kk|h \tanh |\kk|h [t_2-|\kk|\tanh |\kk|h]  \\
& \qquad\mbox{}+\sech |\kk|h \tanh |\kk|h \left[t_1-\frac{\tanh |\kk|h}{|\kk|}\right] +\frac{1}{|\kk|}\sech |\kk|h \tanh^2 |\kk|h, \\
c_{12} &= -\frac{\i}{\alpha} \tanh |\kk|h [s-\sech |\kk|h] +\frac{\i}{\alpha |\kk|} \sech |\kk|h [t_2-|\kk|\tanh |\kk|h],  \\
c_{13} &=- \frac{1}{|\kk|}\sech |\kk|h, \\
c_{21} &= \frac{1}{\alpha} \tanh |\kk|h [s-\sech |\kk|h] +\frac{|\kk|}{\alpha} \sech |\kk|h \left[t_1-\frac{\tanh |\kk|h}{|\kk|}\right], \\
c_{22} & = \frac{\i}{|\kk|}[s-\sech |\kk|h] + \frac{\i}{|\kk|}\sech |\kk|h, \\
c_{23} & = 0, \\
c_{31} & = \frac{\i}{|\kk|}\sech |\kk|h, \\
c_{32} & = 0, \\
c_{33} & = 0,
\end{align*}
where
$$s = \left\{\begin{array}{ll}\sec (\alpha^2-|\kk|^2)^{\frac{1}{2}} h, & |\kk| \leq \alpha, \\[2mm]
\sech (|\kk|^2-\alpha^2)^{\frac{1}{2}} h, & |\kk| \geq \alpha, 
\end{array}\right.$$
$$
t_1 = \left\{\begin{array}{ll}\dfrac{\tan (\alpha^2-|\kk|^2)^{\frac{1}{2}}h}{(\alpha^2-|\kk|^2)^{\frac{1}{2}}}, & |\kk| \leq \alpha, \\[2mm]
\dfrac{\tanh (|\kk|^2-\alpha^2)^{\frac{1}{2}}h}{(|\kk|^2-\alpha^2)^{\frac{1}{2}}}, & |\kk| \geq \alpha,
\end{array}\right.\qquad
t_2 = (|\kk|^2-\alpha^2)t_1;
$$
the necessity that $s$ is finite is met by choosing $\alpha^\star h < \frac{\pi}{2}$, so that
$(\alpha^2-|\kk|^2)^\frac{1}{2}h<\frac{\pi}{2}$ for all $|\kk| \leq \alpha$.

One finds that
\begin{align*}
|s -\sech |\kk|h| & \lesssim \frac{\alpha^2}{|\kk|}\sech |\kk|h, \\
\left|s - \sech |\kk|h + \frac{\alpha^2h}{2|\kk|}\sech |\kk|h \tanh |\kk|h\right| & \lesssim \frac{\alpha^4}{|\kk|^2}\sech |\kk|h, \\
\left| t_1 - \frac{\tanh |\kk|h}{|\kk|} \right| & \lesssim \frac{\alpha^2}{|\kk|^2}, \\
|t_2 - |\kk|\tanh |\kk|h| & \lesssim \alpha^2
\end{align*}
(uniformly over $|\kk| \geq \max(1,\sqrt{2}\alpha)$), and hence
\begin{equation}
|c_{mn}| \lesssim \frac{1}{|\kk|}\sech |\kk|h \label{Appendix est 4}
\end{equation}
uniformly over $|\alpha| \in [0,\alpha^\star]$ and $|\kk| \geq \max(1,\sqrt{2}\alpha)$.
Noting that
$$s \rightarrow \sec \alpha h, \qquad t \rightarrow \frac{\tan \alpha h}{\alpha}$$
and
$$\sech |\kk| h \rightarrow 1, \qquad \frac{\tanh |\kk|h}{|\kk|} \rightarrow h$$
as $|\kk| \rightarrow 0$, we find that
$$|\kk| C \rightarrow \begin{pmatrix} - \frac{1}{2} & - \i \tan \alpha h & -1 \\ 0 & \i \sec \alpha h & 0 \\ \i & 0 & 0 \end{pmatrix}$$
componentwise as $|\kk| \rightarrow 0$; it follows that
\begin{equation}
|\kk|c_{mn} \lesssim 1 \label{Appendix est 5}
\end{equation}
as $|\kk| \rightarrow 0$
(uniformly over $|\alpha| \in [0,\alpha^\star]$).\pagebreak

\dataccess{This paper has no additional data.}\vspace{-0.5\baselineskip}
\aucontribute{The paper was written by M.D.G., and all calculations were checked by J.H.}\vspace{-0.5\baselineskip}
\competing{We declare we have no competing interests.}\vspace{-0.5\baselineskip}
\funding{No external funding was involved in this work.}

\end{document}